\documentclass[12pt]{amsart}
\usepackage{amsmath}
\usepackage{amssymb}
\usepackage{enumerate}
\usepackage{amscd}

\newtheorem{theorem}{Theorem}[section]
\newtheorem{lemma}[theorem]{Lemma}
\newtheorem{proposition}[theorem]{Proposition}
\newtheorem{corollary}[theorem]{Corollary}
\newtheorem{claim}[theorem]{Claim}

\theoremstyle{definition}
\newtheorem{definition}[theorem]{Definition}
\newtheorem{definitions}[theorem]{Definitions}
\newtheorem{example}[theorem]{Example}

\newtheorem{definitions and remarks}[theorem]{Definitions and Remarks}

\theoremstyle{remark}
\newtheorem{remark}[theorem]{Remark}
\newtheorem{remarks}[theorem]{Remarks}

\numberwithin{equation}{section}

\newcommand{\Supp}{\mathrm{Supp}\,}
\newcommand{\supp}{\mathrm{supp}}

\newcommand{\mon}{\mathrm{mon}}
\newcommand{\ord}{\mathrm{ord}}

\newcommand{\length}{\mathrm{length}\,}

\newcommand{\Ex}{\mathrm{Ex}}

\newcommand{\al}{{\alpha}}

\newcommand{\s}{{\sigma}}

\newcommand{\IN}{{\mathbb N}}
\newcommand{\IQ}{{\mathbb Q}}
\newcommand{\IA}{{\mathbb A}}

\newcommand{\IK}{{\mathbb K}}

\newcommand{\cN}{{\mathcal N}}
\newcommand{\cO}{{\mathcal O}}

\newcommand{\cV}{{\mathcal V}}

\newcommand{\fm}{{\mathfrak m}}
\newcommand{\fn}{{\mathfrak n}}

\newcommand{\tX}{{\widetilde X}}
\newcommand{\tD}{{\widetilde D}}
\newcommand{\tE}{{\widetilde E}}

\newcommand{\wA}{{\widehat A}}

\newcommand{\llbracket}{{[\![}}
\newcommand{\rrbracket}{{]\!]}}

\newcommand{\red}[1]{#1}

\begin{document}
\title[Resolution preserving semi-simple normal crossings]{Resolution of singularities of pairs preserving semi-simple normal crossings}

\author{Edward Bierstone}
\address{The Fields Institute, 222 College Street, Toronto, ON, Canada
M5T 3J1, and University of Toronto, Department of Mathematics, 40 St. George Street,
Toronto, Ontario, Canada M5S 2E4}
\email{bierston@fields.utoronto.ca}
\thanks{Research supported in part by NSERC grants OGP0009070 and MRS342058.}

\author{Franklin Vera Pacheco}
\address{University of Toronto, Department of Mathematics, 40 St. George Street,
Toronto, Ontario, Canada M5S 2E4}
\email{franklin.vp@gmail.com}

\subjclass{Primary 14E15, 32S45; Secondary 14C20, 32S10}

\keywords{resolution of singularities, simple normal crossings, semi simple normal crossings, desingularization invariant, Hilbert-Samuel function}

\begin{abstract}
Let $X$ denote a reduced algebraic variety and $D$ a Weil divisor on $X$. The pair
$(X,D)$ is said to be \emph{semi-simple normal crossings} (\emph{semi-snc}) at 
$a\in X$ if $X$ is simple normal crossings at $a$ (i.e., a simple normal 
crossings hypersurface, with 
respect to a local embedding in a smooth ambient variety),
and $D$ is induced by the restriction to $X$ of a hypersurface that is simple normal crossings with respect to $X$. We construct a composition of blowings-up 
$f:\tX\rightarrow X$ such that the transformed pair $(\tX,\tD)$ is everywhere 
semi-simple normal crossings, and $f$ is an isomorphism over the semi-simple normal crossings locus of 
$(X,D)$. The result answers a question of Koll\'ar.
\end{abstract}

\maketitle
\setcounter{tocdepth}{1}
\tableofcontents

\section{Introduction}\label{sec:intro} The subject of this article is partial resolution of singularities
of a pair $(X,D)$, where $X$ is a reduced algebraic variety 
defined over a field of characteristic
zero and $D$ is a Weil $\IQ$-divisor on $X$. 

The purpose of partial resolution of singularities is to provide representatives of a birational
equivalence class that have mild singularities --- almost as good as smooth
--- which have to be admitted in natural situations, even if they can be
eliminated by normalization. For example, in order
to simultaneously resolve the singularities of curves in a parametrized family, one
needs to allow special fibers that have simple normal crossings singularities. Likewise, log
resolution of singularities of a divisor produces a divisor with simple normal crossings.
For these reasons, it is natural to consider simple normal crossings singularities
as acceptable from the start, and to seek a partial resolution which is an isomorphism
over the simple normal crossings locus. 

Our main theorem (Theorem \ref{thm:main})
is a solution of a problem of J\'anos Koll\'ar \cite[Problem 19]{Kolog} on resolution of singularities
of pairs $(X,D)$ except
for \emph{semi-simple normal crossings (semi-snc)} singularities.

\begin{definition}\label{def:ssnc}
Following Koll\'ar, we say that $(X,D)$ is \emph{semi-snc} at a point $a \in X$ if $X$ has a
neighborhood $U$ of $a$ that can be embedded in a smooth variety $Y$, where $Y$ has
regular local coordinates $(x_1,\ldots, x_p, y_1,\ldots,y_r)$ at $a=0$ in which $U$ is defined by a monomial equation
\begin{equation}\label{eq:snc}
x_1\dotsm x_p = 0
\end{equation}
and
\begin{equation}\label{eq:D}
D = \sum_{i=1}^r \al_i(y_i = 0)|_U, \quad \al_i \in \IQ.
\end{equation}
We say that $(X,D)$ is \emph{semi-snc} if it is semi-snc at every point of $X$.
\end{definition}

According to Definition \ref{def:ssnc}, the support, $\Supp D|_U$, of $D|_U$ as a subset of $Y$
is defined by a pair of monomial equations
\begin{equation}\label{eq:Dvar}
x_1\dotsm x_p = 0, \quad y_{i_1} \dotsm y_{i_q} = 0.
\end{equation}

Let $f: \tX \to X$ be a birational mapping. Denote by $\Ex(f)$ the exceptional set of $f$ (i.e. the set of points where $f$ is not a local isomorphism). Assuming that $\Ex(f)$ is a divisor we define $\tD := D' + \Ex(f)$, where $D'$ is
the birational transform of $D$ by $f^{-1}$. We call $(\tX,\tD)$ the \emph{(total) transform} of $(X,D)$ by $f$. 

\begin{theorem}[Main theorem]\label{thm:main}
Let $X$ denote a reduced algebraic variety over a field of characteristic zero, and $D$ a 
Weil $\IQ$-divisor on $X$. Let $U\subset X$ be the largest open subset such that $(U,D|_{U})$ is semi-snc. Then there is a morphism $f:\tX \to X$ given by a composite of blowings-up with smooth (admissible) centers, such that
\begin{enumerate}
\item $(\tX,\tD)$ is semi-snc;
\item\label{conditiontwo} $f$ is an isomorphism over $U$.
\end{enumerate}
\end{theorem}
\begin{remarks}\label{rem:main}
(1) We say that a blowing-up (or its center) is \emph{admissible} if its center is smooth
and has simple normal crossings with respect to the exceptional divisor.
\smallskip

(2) In the special case that $X$ is smooth, we say that $D$ is a  \emph{simple normal crossings} or \emph{snc} divisor on $X$ if $(X,D)$ is semi-snc (i.e., Definition \ref{def:ssnc} is satisfied with $p=1$ at every point of X). This means that the irreducible components of $D$ are smooth and intersect transversely. Theorem \ref{thm:main}, in this case, will be called \emph{snc-strict log resolution} --- this means \emph{log resolution of singularities} of $D$ by a morphism that is an isomorphism over the snc locus (see Theorem \ref{thm:theoremB} below). The latter is proved in \cite[Thm.\,3.1]{BMmin}. Earlier versions can be found in \cite{Sz}, \cite[Sec.\, 12]{BMinv} and \cite{Kolog}. 

Theorem \ref{thm:main} in the special case that $D=0$ also follows from the 
earlier results; see Theorem \ref{thm:snc} below. Both Theorems \ref{thm:snc} and \ref{thm:theoremB} are important ingredients in the proof of Theorem \ref{thm:main}. Theorem \ref{thm:snc} is used to reduce Theorem \ref{thm:main} to the case that $X$ has only snc singularities. When $X$ has only snc singularities Theorem \ref{thm:theoremB} is used to begin an induction on the number of components of $X$.
\smallskip

(3) The desingularization morphism of Theorem \ref{thm:main} is functorial 
in the category of algebraic varieties over a field of characteristic zero 
with a fixed ordering on the components, and with respect to \'{e}tale (or 
smooth) morphisms that preserve the number of irreducible components of $X$ and
$D$ passing through every point. See Section \ref{sec:functoriality}. 
Note that a desingularization that avoids semi-snc and in particular snc 
points cannot be functorial with respect to \'{e}tale morphisms in general 
(as is the case for functorial resolution of singularities), because a normal 
crossings point becomes snc after an \'{e}tale morphism; see Definitions \ref{def:snc} and
Remark \ref{rem:funct}. (Non-snc are to be eliminated while snc are to be preserved.) Therefore we must restrict functoriality to a smaller class of morphisms. 
\smallskip

(4) Theorem \ref{thm:main} holds also with the following 
stronger version of condition \ref{conditiontwo}: The morphism
$f$ is a composite $\sigma_1\circ\ldots\circ\sigma_t$ of blowings-up $\s_i$, where
each $\s_i$ is an isomorphism over the semi-snc locus of the transform of $(X,D)$ by 
$\sigma_1\circ\ldots\circ\sigma_{i-1}$. Our proof provides this stronger statement,
by using a stronger version of log resolution, where every blowing
up is an isomorphism over the snc locus of the preceding transform of $D$. The latter
strong version of log resolution is proved in \cite{BDV} and in \cite[Sect.\,12]{BMinv}.
\end{remarks}

Our approach to partial resolution of singularities is based on the idea developed in \cite{BMmin} and \cite{BLM} that the 
desingularization invariant of \cite{BMinv} together with natural geometric information
can be used to characterize and compute local normal forms of mild singularities. The local normal forms in the latter involve monomials in exceptional divisors that can be simplified or \emph{cleaned} by desingularization of invariantly defined monomial marked ideals. These ideas are 
used in \cite{BMmin} and \cite{BDV} in the proofs of log resolution by a morphism which is
an isomorphism over the snc locus, and are also used in \cite{BMmin, BLM} to treat other problems
stated in \cite{Kolog}, where one wants to find a class of singularities that have to be admitted
if \emph{normal crossings} singularities in a weaker
local analytic or formal sense are to be preserved. 

In \cite{BMmin} and \cite{BDV}, the mild singularities (for example, simple normal crossings
singularities) are all singularities of a \emph{hypersurface} (see definition \ref{def:hypersurface}). The desingularization invariant for a hypersurface is
simpler than for general varieties because it begins with the \emph{order} at a point,
rather than with the \emph{Hilbert-Samuel function}, as in the general case.
Semi-simple normal crossings
singularities (Definition \ref{def:ssnc}) cannot be described as singularities of  a hypersurface
in an ambient smooth variety. An essential feature of this article is our use of the
Hilbert-Samuel function and the desingularization invariant based on it to characterize
semi-snc singularities. 

The results in this article form part of Franklin Vera Pacheco's Ph.D. thesis at the University of
Toronto. The authors are grateful to S\'andor Kov\'acs for explaining some of the 
motivation of the problems considered.

\section{Characterization of semi-snc points}\label{subsec:char}
The inductive characterization of semi-snc (Propostion \ref{lemmafactorssnc} 
below)
will be used after reduction of the main problem to the case that $X$ is an snc hypersurface, no component of $D$ lies in the singular locus of $X$, and $D$ is
reduced. (See \S\ref{subsec:structure} and Section \ref{sec:fixingmultiplicities}.)
Under the preceding assumptions, the main theorem is proved by induction
on the number of components of $X$, and Propostion \ref{lemmafactorssnc} is
used in the inductive step.

Propostion \ref{lemmafactorssnc} applies to points lying in at least two 
components of $X$. The inductive criterion begins with the case of a single
component. In this case, semi-snc means snc. Snc points can be 
characterized using the desingularization invariant \cite[Lemma 3.5]{BMmin}.
We begin by recalling the latter.

\begin{remark}[Characterization of snc singularities]\label{rem:characterizationsnc}
Let $D$ be a reduced Weil divisor on a smooth variety $X$. Assume that $a\in \Supp(D)$ lies in exactly $q$ irreducible components of $D$. Then $D$ is snc at $a$ if and only if the value of the desingularization invariant is $(q,0,1,0,\ldots,1,0,\infty)$, where there are $q-1$ pairs $(1,0)$. (This is in ``year zero'' --- before any blowings-up given by the
desingularization algorithm.)

The first entry of the invariant at a point $a$ of a hypersurface $D$ in a smooth variety
is the order $q$ of $D$ at $a$.  For a subvariety in general, the Hilbert-Samuel function is the first entry of the invariant. (In the case of a hypersurface, the order and the Hilbert-Samuel function each determine the other; see \cite[Remark 1.3]{BMinv} and Section \ref{sec:HS}.)
 \end{remark}
 
\begin{definition}\label{def:Hpq}
Let $H_{p,q}=H_{p,q,n}$ denote the Hilbert-Samuel function of the ideal $(x_1\dotsm x_p,y_1\dotsm y_q)$ in a ring of formal power series $\IK\llbracket x_1,\ldots,x_p,
y_1,\ldots,y_{n-p}\rrbracket$, where $p+q \leq n$. (See Section \ref{sec:HS}.)
\end{definition}

The $H_{p,q}$ are precisely the values that the Hilbert-Samuel function of $\Supp D$
can take at semi-snc points. We will omit the $n$ since it will be fixed throughout the arguments using $H_{p,q}$.

\begin{definition}\label{def:sigmapq}
Assume that $X$ is snc and that $D$ has no components in the singular locus of $X$. We define $\Sigma_{p,q}=\Sigma_{p,q}(X,D)$ as the set of points $a \in X$ such that $a$ lies
in exactly $p$ components of $X$, and $q$ is the minimum number of components of $D$ 
at $a$ which lie in any component of $X$. 
\end{definition}

For example, if $X:=(x_1x_2=0)$ and $D=(x_1=y_1=0)+(x_2=y_1y_2=0)$, then the origin is in $\Sigma_{2,1}$.\medskip
 
Having Hilbert-Samuel function $= H_{p,q}$ at a point of $\Sigma_{p,q}$ is a
necessary condition for semi-snc. But it is not sufficient, even for $(p,q) = (2,1)$,
as we will see in Example \ref{example2}. Additional geometric data is needed. This will be given using an ideal sheaf that is a final obstruction to semi-snc. Blowing up to remove this obstruction involves transformations analogous to the cleaning procedure of \cite[Section 2]{BMmin}, see Proposition \ref{propositioncasepequal2}.

Lemma \ref{lem:HSlemma} below, used in the proof of Propostion \ref{lemmafactorssnc}, provides some initial control over the divisor $D$ at a point
of $\Sigma_{p,q}$ (or $\Sigma_{q,p}$) where the Hilbert-Samuel function has the \emph{correct}
value $H_{p,q}$, provided that $p\geq 2$. 
 
\begin{definition}\label{def:pairs}
Consider a pair $(X,D)$, where $X$ is snc and no component of $D$ lies in the singular locus of $X$. Let $X_1$, \ldots, $X_m$ denote the irreducible components of $X$, with a given ordering. Let $X^i:=X_1\cup\ldots\cup X_i$, $1\leq i\leq m$. Let 
$D_i$ denote the sum of all components of $D$ lying in $X_i$; i.e. $D_i$ is the divisorial part of the restriction of $D$ to $X_i$. We will sometimes write $D_i=D|_{X_i}$. Let $D^i:=\sum_{j=1}^iD_i$.
\end{definition}  
 
\begin{definition}\label{definitionJ}
Consider a pair $(X,D)$ as in Definition \ref{def:pairs}, where $X$ is 
(locally) an embedded hypersurface in a smooth variety $Y$. Assume that 
$m\geq 2$. Let $J=J(X,D)$ 
denote the quotient ideal
 \[
 J=J(X,D):=[I_{D_m}+I_{X^{m-1}}:I_{D^{m-1}}+I_{X_m}],
 \]
where $I_{D_m}$, $I_{X^{m-1}}$, $I_{D^{m-1}}$ and $I_{X_m}$ are the defining 
ideal sheaves of $\Supp D_m$, $X^{m-1}$, $\Supp D^{m-1}$ and $X_m$ 
(respectively) on $Y$.
\end{definition}

 \begin{proposition}[Characterization of semi-snc points.]\label{lemmafactorssnc}
Consider a pair $(X,D)$, where $X$ is (locally) an embedded hypersurface in a smooth variety $Y$. Assume that $X$ is snc, $D$ is reduced and none of the components of 
$D$ lie in the singular locus of $X$. Let $a\in X$ be a point lying in at least two components of $X$. Then $(X,D)$ is semi-snc at $a$ if and only if
 \begin{enumerate}
 \item $(X^{m-1},D^{m-1})$ is semi-snc at $a$.
 \item There exist $p$ and $q$ such that $a\in\Sigma_{p,q}$ and $H_{\Supp D,a}
 =H_{p,q}$, where $H_{\Supp D,a}$ is the Hilbert-Samuel function of $\Supp D$ at the point $a$ and $H_{p,q}$ is defined as in \ref{def:Hpq}.
 \item $J_{a}=\mathcal{O}_{Y,a}$.
 \end{enumerate}
 \end{proposition}
 
 Proposition \ref{lemmafactorssnc} will be proved at the end of Section \ref{sec:HS}.

 \begin{remarks} (1) If $a$ lies in a single component of $X$, then Condition (1) is vacuous and $J$ is not defined. In this case, Remark \ref{rem:characterizationsnc} replaces Lemma
\ref{lemmafactorssnc}.
\smallskip

(2) We will use Proposition \ref{lemmafactorssnc} to remove unwanted 
singularities at points lying in more than two components of $X$,  
by first blowing up to ensure condition (2), and then applying further blowings-up to get condition (3); see Section \ref{sec:morethantwocomponents}. 
At points lying in two components of $X$, it is simpler to control the behavior of $J(X,D)$ after admissible blowings-up; see Section \ref{sec:twocomponents}.
In this case, condition (3) is obtained by a sequence of blowings-up that are very easily described (Proposition \ref{propositioncasepequal2}). 
\end{remarks}


\section{Basic notions and structure of the proof}\label{sec:prelim}

\begin{definition}\label{def:hypersurface}
We say that $X$ is a \emph{hypersurface} at a point $a$ if, locally at $a$, $X$ can
be defined by a principal ideal on a smooth variety.
\end{definition}

\begin{definitions}[cf. Remark \ref{rem:main}(1)]\label{def:snc} 
Let $X$ be an algebraic variety over a field of characteristic zero, and $D$ a Weil $\IQ$-divisor on $X$. The pair $(X,D)$ is said to be \emph{simple normal crossings (snc)} at a closed point
$a \in X$ if $X$ is smooth at $a$ and there is a regular coordinate neighborhood $U$ of $a$ with a system of coordinates $(x_1,x_2,\ldots,x_n)$ such that $\Supp D|_U =(x_1x_2\ldots x_k=0)$, for some $k\leq n$ (or perhaps $\Supp D|_U = \emptyset$). 
Clearly, the set of snc points is open in 
$X$. The \emph{snc locus} of $(X,D)$ is the largest subset of $X$ on which 
$(X,D)$ is snc. The pair $(X,D)$ is \emph{snc} if it is snc at every point of $X$.

Likewise, we will say that an algebraic variety $X$ is \emph{simple normal crossings (snc)}
at $a\in X$ if there is a neighbourhood $U$ of $a$ in $X$ and a local embedding 
$X|_U \stackrel{\iota}{\hookrightarrow}Y$, where $Y$ is a smooth variety, such that $(Y,X|_U)$ is simple normal crossings at $\iota(a)$. (Thus, if $X$ is snc at $a$, then $X$ is a \emph{hypersurface} at $a$.)

The pair $(X,D)$ is called \emph{normal crossings (nc)} at $a\in X$ if there is an \'{e}tale morphism $f:U\rightarrow X$ and a point $b \in U$ such that $a=f(b)$ and $(U,f^{*}(D))$ is snc at $b$. 

If $D=\sum a_iD_i$, where $D_i$ are prime divisors, then $D_\text{red}$ denotes $\sum D_i$, i.e. $D_\text{red}$ is $\Supp D$ considered as a divisor.
\end{definitions}

\begin{example}\label{ex:node}
The curve $X:=(y^2+x^2+x^3=0)\subset\mathbb{A}^2$ is nc but not snc at $0$. It is not snc because it has only one irreducible component which is not smooth at $0$. But $X$ is
nc at $0$ because $X$ has two analytic branches at $0$ which intersect transversely.
\end{example}

It is important to distinguish between nc and snc. For example, the analogues for nc of
log resolution preserving the nc locus or of Theorem \ref{thm:main} are false:

\begin{example}\label{ex:pp}
Consider the pair $(\mathbb{C}^3, D)$, where $D= (x^2-yz^2=0)$. The singularity at $0$ is 
called a \emph{pinch point}. The pair is nc at every point except the origin. The analogue of Theorem
\ref{thm:main} for nc fails in this example because we cannot get rid of the pinch point without blowing up the $y$-axis, according to the following argument of Koll\'{a}r 
\cite[Ex.\,8]{Kolog} (see also Fujino \cite[Cor.\,3.6.10]{Fuji}). The hypersurface $D$ has two sheets over every non-zero point of $(z=0)$. Going around the origin in $(z=x=0)$ permutes the sheets, and this phenomenon persists after any birational morphism which is an isomorphim over the generic point of $(z=x=0)$.
\end{example}

\begin{definitions}\label{def:transf}
If $f:X\rightarrow Y$ is a rational mapping and $Z\subset X$ is a subvariety such that $f$ is defined in a dense subset $Z_0$, then we define the \emph{birational transform} $f_{*}(Z)$
of $Z$ as the closure of $f(Z_0)$ in $Y$. In the case that $f$ is birational, then we have the notion of $f_{*}^{-1}(Z)$ for subvarieties $Z\subset Y$ such that $f^{-1}$ is defined in a dense 
subset of $Z$. For a divisor $D=\sum \alpha_i D_i$, where the $D_i$ are prime divisors, 
we define $f_{*}^{-1}(D):=\sum\alpha_i f_{*}^{-1}(D_i)$.

If $f:X\rightarrow Y$ is a birational mapping, we let $\Ex(f)$ denote the set of points $a \in X$ 
where $f$ is not biregular; i.e., $f^{-1}$ is not a morphism at $f(a)$. We consider $\Ex(f)$ with the structure of a reduced subvariety of $X$. 

As before, consider $(X,D)$, where $X$ is an algebraic variety $X$ over a field of characteristic zero and $D$ is a Weil divisor. Let $f:\tX\rightarrow X$ be a proper birational map and \emph{assume that $\Ex(f)$ is a divisor}. Then we define 
\[
D' := f_{*}^{-1}(D)\quad\text{ and }\quad\tD := D' +\Ex(f).
\]
We call $D'$ the \emph{strict} or \emph{birational transform} of $D$ by $f$, and we call
$\tD$ the \emph{total transform} of $D$.  We also call $(\tX,\tD)$ the \emph{(total) transform} 
of $(X,D)$ by $f$.
\end{definitions}

\begin{remark}\label{rem:transf}
It will be convenient to treat $D'$ and $\Ex(f)$ separately in our proof of Theorem \ref{thm:main} --- we need to count the components of $D'$ rather than of $\tD$. For this reason, we will work with data given by a triple $(X,D,E)$, where initially $(X,D)$ is the given pair and $E=\emptyset$. After a blowing-up $f:X'\rightarrow X$, we will consider the transformed data given by $(X',D',\tE)$, where $D':=f_{*}^{-1}(D)$ as above and 
$\tE:=f_{*}^{-1}(E)+\Ex(f)$. 

We will write $f:(X',D')\rightarrow (X,D)$ to mean that $f:X'\rightarrow X$ is birational and $D'$ is the strict transform of $D$ by $f$. 
\end{remark}

\begin{definition}\label{def:triplessnc}
We say that a triple $(X,D,E)$, where $D$ and $E$ are both divisors on $X$, is 
\emph{semi-snc} if $(X, D+E)$ is semi-snc (see Definition \ref{def:ssnc}).
\end{definition}

For economy of notation, when there is no possibility of confusion, we will sometimes 
denote the transform of $(X,D,E)$ by a
sequence of blowings-up still simply as $(X,D,E)$. Other constructions depending on $X$ and $D$ are also denoted by symbols that will be preserved after transformation by blowings-up. This convention is convenient for the purpose of describing an algorithm,
and imitates computer programs written in imperative languages. 

\begin{example}
 Consider $(X,D)$, where $X=(x_1^2-x_2^2x_3=0)\subset\mathbb{A}^3$ and $D=(x_1=x_3=0)$. Let $f$ denote the blowing-up of $\IA^3$ with center the $x_3$-axis. Then, the strict transform $X' = \tX$ of $X$ by $f$ (i.e., the blowing-up of $X$ with center the 
 $x_3$-axis) lies in one chart of $f$ (the ``$x_2$-chart'') with coordinates $(y_1,y_2,y_3)$ in which $f$ is given by
\[
x_1=y_1y_2, \quad x_2=y_2, \quad x_3=y_3.
\]
Therefore we have $\tX =(y_1^2-y_3=0)$ and $\tD =f^{-1}_\ast (D)+E$, where $E$ is the exceptional divisor; $E=(y_1^2-y_3=y_2=0)$. Then
\begin{align*}
 \tD=&(y_1=y_3=0)+(y_1^2-y_3=y_2=0)\\
   =&(y_1=y_1^2-y_3=0)+(y_1^2-y_3=y_2=0).
\end{align*}
We see that, at the origin in the system of coordinates $z_1:=y_1$, $z_2:=y_2$, $z_3:=y_3-y_1^2$, the pair $(\tX,\tD)$ is given by $\tX=(z_3=0)$, $\tD=(z_3=y_1=0)+(z_3=y_2=0)$, and is therefore snc.
\end{example}

\begin{example}\label{ex:multiplicities}
If $X=(xy=0)\subset Y:=\mathbb{A}^3$ and $D=a_1D_1+a_2D_2$, where $D_1=(x=z=0)$ and $D_2=(y=z=0)$, then the pair $(X,D)$ is semi-snc if and only if $a_1=a_2$.
\end{example}

At a semi-snc point, the local picture is that $X$ is a snc hypersurface in a smooth
variety $Y$, and $D$ is given by the intersection of $X$ with a snc divisor in $Y$ which is
transverse to $X$ (in Example \ref{ex:multiplicities}, $(z=0)$). For this reason, we should have the same multiplicities when one component of this divisor intersects different components of $X$.

\subsection{Structure of the proof}\label{subsec:structure}
The desingularization morphism from Theorem \ref{thm:main} is a composition of blowings-up with smooth centers. In the rest of the paper, $(X,D)$ will always denote a pair satisfying the assumptions of Theorem \ref{thm:main}. Our proof of the theorem involves an algorithm for successively choosing the centers of blowings-up, that will be described precisely in section \ref{sec:maintheorem}. We will give an idea of the main ingredients in the current subsection. As noted in Remark \ref{rem:main} (2), the following two theorems are previously known special cases of our main result that are used in its proof. 

\begin{theorem}[snc-strict desingularization]\label{thm:snc}\label{cor:corollarytheoremB}
Let $X$ denote a reduced scheme of finite type over a field of characteristic zero. Then, there is a finite sequence of blowings-up with smooth centers 
\begin{equation}\label{eqn:blowupsequence}
X:=X_0\stackrel{\sigma_1}{\longleftarrow} X_1\stackrel{\sigma_2}{\longleftarrow}\ldots\stackrel{\sigma_t}{\longleftarrow} X_t=:\tX,
\end{equation}
such that, if $\tD$ denotes the exceptional divisor of \eqref{eqn:blowupsequence}, then $(\tX,\tD)$ is semi-snc and $(X,0)\leftarrow(\tX,\tD)$ is an isomorphism over the snc-locus $X^{snc}$ of $X$.
\end{theorem}

Theorem \ref{thm:snc} can be strengthened so that, not only is $\tX\rightarrow X$ an isomorphism over the snc locus of $X$ but also $\sigma_{k+1}$ is an isomorphism over the semi-snc points of  $(X_k,D_k)$, where $D_k$ is the exeptional divisor of $\sigma_{1}\circ\ldots\circ\sigma_{k}$, for every $k=0,\ldots,t-1$. (See \cite{BDV}; cf. Remarks \ref{rem:main}(4)).

\begin{theorem}[snc-strict log resolution {\cite[Thm.\,3.1]{BMmin}}]\label{thm:theoremB}
Consider a pair $(X,D)$, as in Theorem \ref{thm:main}. Assume that $X$ is smooth. Then there is a finite sequence of blowings-up with smooth centers over the support of $D$ (or its strict transforms)
\[
X:=X_0\stackrel{\sigma_1}{\longleftarrow} X_1\stackrel{\sigma_2}{\longleftarrow}\ldots\stackrel{\sigma_t}{\longleftarrow} X_t=:\tX,
\]
such that the (reduced) total transform of $D$ is snc and $X\leftarrow \tX$ is an isomorphism over the snc locus of $(X,D)$.
\end{theorem}
\begin{remark}
Theorems \ref{thm:snc} and \ref{thm:theoremB} are both functorial in the sense of Remark \ref{rem:main}(3). Moreover, regarding $D$ as a hypersurface in $X$, the blow-up sequence for $D$ is independent of the embedding space $X$. Theorem \ref{thm:snc} follows from functoriality in Theorem \ref{thm:theoremB}.
\end{remark}

\begin{proof}[Proof of Theorem \ref{thm:snc}]
We can first reduce Theorem \ref{thm:snc} to the case that $X$ is a hypersurface: If $X$ is of pure dimension, this reduction follows simply from the strong desingularization algorithm of \cite{BMinv,BMfunct}. The algorithm involves blowing up with smooth centers in the maximum strata of the Hilbert-Samuel function $H_{X,a}$. The latter determines the local embedding dimension $e_{X}(a):=H_{X,a}(1)-1$, so the algorithm first eliminates points of embedding codimension $>1$ without modifying nc points.

When $X$ is not of pure dimension the desingularization algorithm \cite{BMmin,BMfunct} may involve blowing up hypersurface singularities in higher dimensional components of $X$ before $X$ becomes a hypersurface everywhere. This problem can be corrected by a modification of the desingularization invariant described in \cite{BMT}: 

Let $\#(a)$ denote the number of different dimensions of irreducible components of $X$ at $a\in X$. Let $q(a)$ be the smallest dimension of an irreducible component of $X$ at $a$ and set $d:=dim(X)$. Then, instead of using the Hilbert-Samuel function as first entry of the invariant, we use the pair $\phi(a):=(\#(a),H_{X\times\mathbb{A}^{d-q(a)},(a,0)})$.

The original and modified invariants admit the same local presentations (in the sense of \cite{BMinv}). This implies that every component of a constant locus 
of one of the invariants is also a component of a constant locus of the other. The modification ensures that the irreducible components of the maximal locus of the usual invariant are blown up in a convenient order rather that at the same time. Since the modified invariant begins with $\#(a)$, points where there are components of different dimensions will be blown up first. Points with $\#(a)>1$ are not hypersurface points.

If $\#(a)=\#(b)=1$ and $q(a)<q(b)$, then the adjusted Hilbert-Samuel function guarantees that the point with larger value of 
\[
H_{X\times\mathbb{A}^{d-q(a)}}(1)=e(\cdot)+d-q(\cdot)+1,
\]
where $e=e_{X}$, will be blown up first. In particular, non-hypersurface singularities (where $e(\cdot)-q(\cdot)>1$) will be blown up before hypersurface singularities (where $e(\cdot)-q(\cdot)\leq1$).

We can thus reduce to the case in which $X$ is everywhere a hypersurface. Then 
$X$ locally admits a codimension one embedding in a smooth variety. For each 
local embedding we can apply Theorem \ref{thm:theoremB}. Functoriality in 
Theorem \ref{thm:theoremB} (with respect to embeddings, and \'{e}tale 
morphisms preserving the number of components) can be use to show that the local desingularizations glue together to define global centers of blowing up for $X$ (cf. \cite[proof of Prop.\,3.37]{Ko}).
\end{proof}

We now outline the proof of the main theorem. First, we can use Theorem \ref{thm:snc} to reduce to the case that $X$ is snc; see Section \ref{sec:maintheorem}, Step $1$. Moreover, there is a simple combinatorial argument to reduce to the case that $D$ is a \emph{reduced} divisor (i.e., each $\al_i = 1$ in
Definition \ref{def:ssnc}); see Step $4$ in Section \ref{sec:maintheorem} and Section \ref{sec:fixingmultiplicities}. 

So we can assume that $X$ is snc and $D$ is reduced. We now argue by induction on the number of components of $X$. 

To begin the induction (Section \ref{sec:maintheorem}, Step $3$), we use Theorem \ref{thm:theoremB} to transform the first component of $X$ together with the components of $D$ lying in it, into a semi-snc pair. By induction, we can assume that the pair given by $X$ minus its last component, together with the corresponding restriction of $D$, is semi-snc. (By \emph{restriction} we mean the divisorial part of the restriction of $D$). To complete the inductive step,
we then have to describe further blowings-up to remove the unwanted singularities in the last component of $X$. These blowings-up are separated into blocks which resolve the non-semi-snc singularities in a sequence of strata $\Sigma_{p,q}$ that exhaust the variety; see Definition \ref{def:sigmapq}.

Note first that, in the special case that $X$ is snc, each component of $D$ either lies in 
precisely one component of $X$ (as, for example, if $(X,D)$ is semi-snc) or it is a component of the intersection of a pair of components of $X$ (e.g., if $X:=(xy=0)\subset\mathbb{A}^2$ and $D=(x=y=0)$). We can reduce to the case that each component of $D$ lies in precisely one component of $X$ by blowing up to eliminate components of $D$ that are contained in the singular locus of $X$ (see Section \ref{sec:maintheorem}, Step $2$). Except for this step, our algorithm never involves blowing up with centers of codimension one in $X$.

We remove non-semi-snc singularities iteratively in the strata $\Sigma_{p,q}$, for decreasing values of $(p,q)$. The cases $p=1$, $p=2$ and $p\geq3$ are treated differently. 

In the case $p=1$ the notions of snc and semi-snc coincide, so again we use snc-strict
log resolution (Theorem \ref{thm:theoremB}). The cases $p=2$ and $p\geq3$ will be treated in sections \ref{sec:twocomponents} and \ref{sec:morethantwocomponents}, respectively. All of these cases are part of Step $3$ in Section \ref{sec:maintheorem}.

As remarked in Section \ref{sec:intro}, our approach is based on the idea that the
desingularization invariant of \cite{BMinv} together with natural geometric information
can be used to characterize mild singularities. For snc singularities, it is enough to
use the desingularization invariant for a hypersurface together with the number of irreducible components
at a point, see \cite[\S3]{BMmin}. 

In this article, the main object is a pair $(X,D)$. If $X$ is locally embedded as a hypersurface
in a smooth variety $Y$ (for example, if $X$ is snc), then (the support of) $D$ is of 
codimension two in $Y$. We will need the desingularization invariant for the support of $D$.
The first entry in this invariant is the \emph{Hilbert-Samuel function}
of the local ring of $\Supp D$ at a point (see Section \ref{sec:HS} below). Information
coming from the Hilbert-Samuel function will be used to identify non-semi-snc singularities.

\section{The Hilbert-Samuel function and semi-simple normal crossings}\label{sec:HS}
Lemma \ref{lem:HSlemma} of this section plays an important
part in our use of the Hilbert-Samuel function to characterize semi-snc points.
We begin with the definition of the Hilbert-Samuel function and
its relationship with the diagram of initial exponents (cf. \cite{BMjams}). 
At the end of this section, we use Lemma \ref{lem:HSlemma} to prove the 
inductive characterization of semi-snc (Lemma \ref{lemmafactorssnc}).

\begin{definition}\label{def:HS}
Let $A$ denote a Noetherian local ring $A$ with maximal ideal $\fm$. The \emph{Hilbert-Samuel function} $H_A \in \IN^{\IN}$ of $A$ is defined by
\[
H_A(k) := \length \frac{A}{\fm^{k+1}}, \quad k \in \IN.
\]
If $I \subset A$ is an ideal, we sometimes write $H_{I}:=H_{A/I}$.
If $X$ is an algebraic variety and $a\in X$ is a closed point, we define
$H_{X,a}:=H_{{\cO}_{X,a}}$, where $\cO_{X,a}$ denotes the local ring of $X$ at $a$.
\end{definition}

\begin{definition}\label{definitionorderHSfunction}
Let $f,g \in \mathbb{N}^{\mathbb{N}}$. We say that $f>g$ if $f(n)\geq g(n)$, for every $n$, and  $f(m)>g(m)$, for some $m$. This relation induces a partial order on the set of all possible values for the Hilbert-Samuel functions of Noetherian local rings.
\end{definition}
Note that $f\nleq g$ if and only if either $f>g$ or $f$ is incomparable with $g$.

Let $\wA$ denotes the completion of $A$ with respect to $\fm$. Then $H_{A}=H_{\wA}$, see \cite[\S24.D]{Ma}. If $A$ is regular, then we can identify $\wA$ with a ring of formal power series, $\IK\llbracket x\rrbracket$, where $x=(x_1,\ldots,x_n)$. Then $\fn := (x_1,\ldots,x_n)$
is the maximal ideal of $\IK\llbracket x\rrbracket$. If $I\subset \IK\llbracket x\rrbracket$ is an
ideal, then
\[
H_I(k):=\dim_{\IK}\frac{\IK\llbracket x\rrbracket}{I+\fn^{k+1}}.
\]

If $\alpha=(\alpha_1,\ldots,\alpha_n)\in\mathbb{N}^{n}$, set $|\alpha|:=\alpha_1+\ldots+\alpha_n$. The lexicographic order of $(n+1)$-tuples, $(|\alpha|,\alpha_1,\ldots,\alpha_n)$ induces a total ordering of $\mathbb{N}^n$. Let $f\in \IK\llbracket x\rrbracket$ and write $f=\sum_{\alpha\in\mathbb{N}^n}f_\alpha x^{\alpha}$, where $x^{\alpha}$ denotes $x_1^{\alpha_1}\dotsm x_n^{\alpha_n}$. Define $\supp(f)=\{\alpha\in\mathbb{N}^n:\,f_\alpha\neq0\}$. The \emph{initial exponent} $\exp(f)$ is defined as the smallest element of $\supp(f)$. If $\al=\exp(f)$, then $f_\alpha x^{\alpha}$ is called the \emph{initial monomial} $\mon(f)$ 
of $f$.

\begin{definition}\label{def:diag}. Consider an ideal $I\subset K\llbracket x\rrbracket$. The \emph{initial monomial ideal} 
$\mon(I)$ of $I$ denotes the ideal generated by $\{\mon(f):\ f\in I\}$. The \emph{diagram of initial exponents} $\cN(I)\subset\mathbb{N}^n$ is defined as
\[
\cN(I) := \{\exp(f): f \in I\setminus\{0\}\}.
\]
\end{definition}

Clearly, $\cN(I) + \IN^n = \cN(I)$. For any $\cN \subset\mathbb{N}^n$ such that $\cN=
\cN+\mathbb{N}^n$, there is a smallest set $\cV \subset \cN$ such that $\cN=\cV+\cN$; moreover, $\cV$ is finite. We call $\cV$ the set of \emph{vertices} of $\cN$.

\begin{proposition}\label{prop:diag}
For every $k\in\mathbb{N}$, $H_{I}(k)=H_{\mon(I)}(k)$ is the number of elements $\alpha\in\mathbb{N}^n$ such that $\alpha\notin \cN(I)$ and $|\alpha|\leq k$.
\end{proposition}
\begin{proof}
See \cite[Corollary 3.20]{BMinv}.
\end{proof}

\begin{definition}\label{def:orderSigmapq}
We can use the partial ordering of the set of all Hilbert-Samuel functions to also order the strata 
$\Sigma_{p,q}$ (see Definition \ref{def:sigmapq}). We say that $\Sigma_{p_1,q_1}$ \emph{precedes} $\Sigma_{p_2,q_2}$ if $(\delta(p_1),H_{p_1,q_1})>(\delta(p_2),H_{p_2,q_2})$ in the lexicographic order, where 
\begin{equation*}
\delta(p)=\begin{cases} 3,\text{ if }p\geq 3\\ p\text{ otherwise.}\end{cases} 
\end{equation*}
This order corresponds to the order in which we are going to remove the non-semi-snc from these strata.
\end{definition}

The following two examples illustrate the kind of information we can expect to get from
the Hilbert-Samuel function.

\begin{example}\label{example1}
Let $X: = X_1 \cup X_2$, where $X_1: = (x_1=0)$, $X_2: = (x_2=0)\subset\mathbb{A}_{(x_1,x_2,y,z)}^4$. Note that, if $(X,D)$ is semi-snc, then 
$\Supp D|_{X_1} \cap \Supp D|_{X_2}$ has codimension $2$ in $X$. Consider 
$D:=(x_1=y=0)+(x_2=z=0)$. Then, the origin is not semi-snc. In fact, 
$\Supp D|_{X_1} \cap \Supp D|_{X_2}\,=\,(x_1=x_2=y=z=0)$, which has codimension $3$ in 
$X$. The Hilbert-Samuel function of $\Supp D$ at the origin detects such an anomaly in codimension at a point in a given stratum $\Sigma_{p,q}$ (see Remark \ref{note:HSlemma} and Lemma \ref{lem:strongHSlemma}). In the preceding example, the origin belongs to $\Sigma_{2,1}$ but the Hilbert-Samuel function is not equal to $H_{2,1}$. In fact, the ideal of $\Supp D$ (as
a subvariety of $\mathbb{A}^{4}$) is $(x_1,y)\cap(x_2,z)=(x_1,y)\cdot(x_2,z)$, which has order $2$ while 
$(x_1x_2,y)$, which is the ideal of the support of $D$ at a semi-snc point in $\Sigma_{2,1}$, is of order $1$. The Hilbert-Samuel function determines the order and therefore differs in these two examples.
\end{example}

\begin{example}\label{example2}
This example will show that, nevertheless, the Hilbert-Samuel function together with the number of components of $X$ and $D$ does not suffice to characterize
semi-snc. Consider $X:=(x_1x_2=0)\subset\mathbb{A}_{(x_1,x_2,y,z)}^4$ and $D:=D_1+D_2:=(x_1=y=0)+(x_2=x_1+yz=0)$. Again the origin is  not semi-snc, since the intersection of $D_1$ with $X_2:=(x_2=0)$ and of $D_2$ with $X_1:=(x_1=0)$ are not the same (as they should be at semi-snc points). On the other hand, the Hilbert-Samuel function does not detect the non-semi-snc singularity, since it is the same for the ideals $(x_1,y)\cap(x_2,x_1+yz)$ and $(x_1x_2,y)$. In fact, the Hilbert-Samuel function is determined by the initial monomial ideal of $\Supp D$. 
Since $(x_1,y)\cap(x_2,x_1+yz)=(x_1x_2, x_2y,x_1+yz)$, we compute its initial monomial
ideal as $(x_1,x_2y)$. The latter has the same Hilbert-Samuel function as $(x_1x_2,y)$. This example motivates definition \ref{definitionJ}, which is the final ingredient in our characterization of the semi-snc singularities (Lemma \ref{lemmafactorssnc}).
\end{example}

In Example \ref{example2}, although the intersections of $D_1$ with $X_2$ and of $D_2$ with $X_1$ are not the same, the intersection $D_2\cap X_1$ has the same components as $D_1\cap X_2$ plus some extra components (precisely, plus one extra component $(x_1=x_2=z=0)$). The following lemma shows that this is the worst that can happen when we have the correct value $H_{p,q}$ of the Hilbert-Samuel function in $\Sigma_{p,q}$. 

\begin{lemma}\label{lem:HSlemma}
Assume that $(X,D)$ is locally embedded in a coordinate chart of a smooth variety $Y$ with a system of coordinates $(x_1,\allowbreak\ldots,\allowbreak x_p,\allowbreak y_1,\ldots, y_q,\allowbreak w_1,\ldots, w_{n-p-q})$. Assume $X = (x_1\dotsm x_p = 0)$.
Suppose that $D$ is a reduced divisor \emph{(}so we view it as a subvariety\emph{)}, with no components in the singular locus of $X$, given by an ideal $I_D$
at $a=0$ of the form
\begin{equation}\label{eq:HSlemma}
I_D=(x_1\dotsm x_{p-1},y_1\dotsm y_r)\cap(x_p,f).
\end{equation}
Consider $a\in\Sigma_{p,q}$, where $p\geq2$. \emph{(}In particular $q$ is the minimum of $r$ and the number of irreducible factors of $f|_{(x_p=0)}$\emph{)}. Let $H_D$ denote the Hilbert-Samuel function of $I_D$. 

\emph{Then} $H_D=H_{p,q}$ if and only if we can choose $f$ so that $\ord(f)=q$, $r=q$ and $f\in(x_1\dotsm x_{p-1},y_1\dotsm y_r,x_p)$. Moreover, if either $f\notin(x_1\dotsm x_{p-1},\allowbreak y_1\dotsm y_r,x_p)$, $\ord(f)>q$ or $r>q$ then $H_D\nleq H_{p,q}$ \emph{(}see Definition \ref{definitionorderHSfunction} ff.\emph{)}.
\end{lemma}
\begin{remark}
It follows immediately from the conclusion of the lemma that $H_{D}\not< H_{p,q}$ at a point in $\Sigma_{p,q}$.
\end{remark}

\begin{proof}[Proof of Lemma \ref{lem:HSlemma}]
First we will give a more precise description of the ideal $I_D$. Let $I\subset\{1,2,\ldots,p-1\}\times\{1,2,\ldots,r\}$ denote the set of all $(i,j)$ such that $(x_p,f)+(x_i,y_j)$ defines a subvariety of codimension 3 in the ambient variety $Y$ (i.e. a subvariety of codimension 2 in $X$). For such $(i,j)$, any element in $(x_p,f)$ belongs to the ideal $(x_p,x_i,y_j)$. Set $G:=\bigcap_{(i,j)\in I}(x_i,y_j)$ and $H:=\bigcap_{(i,j)\notin I}(x_i,y_j)$; note that these are the prime decompositions. Then any element of $(x_p,f)$ belongs to $\bigcap_{(i,j)\in I}(x_p,x_i,y_j)=(x_p)+G$. Therefore we can take $f\in G$. Observe that we still have $f\notin (x_i,y_j)$ for $(i,j)\notin I$.

We claim that
\begin{equation}\label{eq:claim1}
G\cap(x_p,f)=(x_p)\cdot G+(f).
\end{equation}

To prove \eqref{eq:claim1}: The inclusion $G\cap(x_p,f)\supset(x_p)\cdot G+(f)$ is clear since
$f\in G$. To prove the other inclusion, consider $a\in G\cap(x_p,f)$. Write $a=fg_1+x_pg_2$. Then $x_pg_2\in G=\bigcap_{(i,j)\in I}(x_i,y_j)$. Since $x_p\notin (x_i,y_j)$, for every $(i,j)\in G$, we have $g_2\in G$. It follows that $a=x_pg_2+fg_1\in(x_p)\cdot G+(f)$, as required.

We now claim that
\begin{equation}\label{eq:claim2}
H\cap[G\cdot(x_p)+(f)]=(x_p)\cdot[G\cap H]+H\cap(f):
\end{equation}

As in the previous claim, the inclusion $H\cap[G\cdot(x_p)+(f)]\supset(x_p)\cdot[G\cap H]+H\cap(f)$ is clear. To prove the other inclusion, consider $a\in H\cap[G\cdot(x_p)+(f)]$. Then $a=fg_1+x_pg\in H$, where $g\in G$. This implies that $fg_1\in(x_p)+H=\bigcap_{(i,j)\notin I}(x_p,x_i,y_j)$. Consider $(i,j)\notin I$. Assume that $f\in (x_p,x_i,y_j)$. Then there is an irreducible factor $f_0$ of $f$, such that $f_0\in(x_p,x_i,y_j)$. If $f_0=x_ph_1+x_ih_2+y_jh_3$ with $h_3\neq0$, then $(x_p,f)+(x_i,y_j)=(x_p,x_i,y_j)$, which contradicts $(i,j)\notin I$. Now, if $h_3=0$, then $f_0=x_ph_1+x_ih_2\in (x_p,x_i)$, which implies $f\in(x_p,x_i)$, contradicting the assumption that $D$ has no component in the singular locus of $X$. Thus $f\notin(x_p,x_i,y_j)$. Since $(x_p,x_i,y_j)$ is prime, it follows that $g_1\in(x_p)+H$ and $g_1=x_pg_{11}+h$, where $h\in H$. Thus $a=fh+x_p(fg_{11}+g)$ and 
therefore $x_p(fg_{11}+g)\in H$. Since $x_p$ is not in any of the prime factors of $H$, it follows that $fg_{11}+g\in H$. Thus $a\in(x_p)\cdot[G\cap H]+H\cap(f)$.

By \eqref{eq:claim1} and \eqref{eq:claim2},
\begin{align}
I_D&=G\cap H\cap (x_p,f)\notag\\
 &=H\cap[G\cdot(x_p)+(f)]\label{HSequation}\\
 &=(x_p)\cdot[H\cap G]+H\cap(f).\notag
\end{align}

We are allowed to pass to the completion of the local ring of $Y$ at $a$ with respect to its maximal ideal. So we can assume we are working in a formal power series ring where $(x_1,\ldots,x_p,y_1,\ldots,y_{n-p})$ are the indeterminates. We can pass to the completion because this doesn't change the Hilbert-Samuel function, the order of $f$ or ideal membership. For simplicity, we
use the same notation for ideals and their generators before and after completion. 

We can compute the Hilbert-Samuel function $H_D$ using the diagram of initial exponents of our ideal $I_D$. This diagram should be compared to the diagram of the ideal $(x_1\dotsm x_p,y_1\dotsm y_q)$, which has exactly two vertices, in degrees $p$ and $q$.

All elements of $H\cap(f)=H\cdot(f)$ have order strictly greater than $\ord(f)$ (which is $\geq q$), unless $H=(1)$ and $\ord(f)=q$. Moreover, all elements of 
\[
(x_p)\cdot[G\cap H]=(x_1\dotsm x_p,\ x_py_1\dotsm y_r)
\]
of order less than $q+1$ have initial monomial divisible by $x_1x_2\dotsm x_p$. 

It follows that, if $f\notin(x_1\dotsm x_{p-1},y_1\dotsm y_r)$ i.e. if $H\neq(1)$, then $H_D\nleq H_{p,q}$. To see this, first assume that $p\geq q+1$. Then all elements of the ideal $I_D=(x_p)\cap[H\cap G]+H\cap(f)$ have order $\geq q+1$,
but $(x_1\dotsm x_p,y_1\dotsm y_q)$ contains an element of order $q$. Therefore $H_D\nleq H_{p,q}$ (obvious from the diagram of initial exponents). Now suppose that $p<q+1$. All elements of $(x_p)\cap[H\cap G]$ of order less than $q+1$ have initial monomials divisible by $x_1\dotsm x_p$, while $y_1\dotsm y_q \in (x_1\dotsm x_p,y_1\dotsm y_q)$ has order $q<q+1$ but its initial monomial is not divisible by $x_1\dotsm x_p$. Therefore we again get $H_D\nleq H_{p,q}$. 

Assume that $\ord(f)>q$. We have just seen that every element of $(x_p)\cap[H\cap G]$ of order $<q+1$ has initial monomial divisible by $x_1\dotsm x_{p-1}$. Therefore every element of $I_D=(x_p)\cap[H\cap G]+H\cdot (f)$ of order $<q+1$ has initial monomial divisible by $x_1\dotsm x_{p-1}$. But, in $(x_1\dotsm x_{p}, y_1\dotsm y_q)$, the element $y_1\dotsm y_q$ has order $q<q+1$ but is not divisible by $x_1\dotsm x_{p-1}$. Therefore $H_D\nleq H_{p,q}$.

If $f\in(x_1\dotsm x_{p-1},y_1\dotsm y_r)$, $\ord(f)=q$ but $r>q$, then the initial monomial of $f$ is divisible by $x_1\dotsm x_{p-1}$. A simple computation shows that the ideal of initial monomials of $I_D$ is 
\[
 (x_1\dotsm x_{p},\allowbreak x_py_1\dotsm y_{r}, \text{mon}(f)).
\]
This follows from the fact that canceling the initial monomial of $f$ using $x_1\dotsm x_p$ or $x_py_1\dotsm y_q$ leads to a function whose initial monomial is already in $(x_1\dotsm x_p,x_py_1\dotsm y_q)$. For convenience, write $a:=x_1\dotsm x_p$, $b:=\mon(f)$ and $c:=x_py_1\dotsm y_q$. From the diagram it follows that $H_{D}\nleq H_{p,q}$ because the monomials that are multiples of both $a$ and of $b$ are not only those that are multiples of $ab$ and therefore $H_{D}(q+1)>H_{p,q}(q+1)$.

It remains to show that if $f\in(x_1\dotsm x_{p-1},y_1\dotsm y_r)$ (i.e., $H=(1)$), $r=q$ and that $\ord(f)=q$ then $H_D=H_{p,q}$. Assume that $H=(1)$ and that $\ord(f)=q$. The first assumption implies that
\begin{equation}\label{eq:ideal}
I_D=(x_1\dotsm x_p,x_py_1\dotsm y_q,f).
\end{equation}

We consider two cases: (1) $p\leq q$. Since $H=(1)$, $f\in G=I_D$.
Therefore, we have one of the following options for the initial monomial of $f$.
\begin{equation}\label{eq:cases}
\mon(f)=\begin{cases}y_1y_2\dotsm y_q&\\x_1\dotsm x_{p-1}\overline{y}\\ x_1\dotsm x_{p-1}\overline{y}z\\x1\dotsm x_{p-1}z,\end{cases}
\end{equation}
where $\overline{y}$ is a product of some of the $y_j$ and $z$ is a product of some of
the remaining coordinates (possibly including some of the $x_i$). (In every case, the degree of the monomial is $q$.)

In each case in \eqref{eq:cases} the ideal of initial monomials of $I_D$ is $(x_1\dotsm x_{p},\allowbreak x_py_1\dotsm y_q,\mon(f))$. 

We want to prove now that, in all cases in \eqref{eq:cases}, $H_{\text{mon}(I_D)}=H_{p,q}$. For convenience, write $a:=x_1\dotsm x_p$, $b:=\mon(f)$ and $c:=x_py_1\dotsm y_q$. In the first case
of \eqref{eq:cases}, the equality is precisely the definition of $H_{p,q}$. Note that, in the remaining cases, the Hilbert-Samuel function of the ideal $(a,b)$ is larger than $H_{p,q}$ because the monomials that are multiples of both $a$ and of $b$ are not only those that are multiples of $ab$. For example, in the second case (i.e., $\mon(f)=x_1\dotsm x_{p-1}\overline{y}$), such monomials are those of the form $a\overline{y}m=bx_pm$ where $m\notin(x_1\dotsm x_{p-1})$. When $\deg(m)=d$, these terms have degree $q+d+1$, but the monomial $x_py_1\dotsm y_qm \in \mon(I_D)$ (of the same degree) does not belong to the ideal $(a,b)$. This implies that the diagrams of initial exponents of the ideals $I_D$ and $(x_1\dotsm x_p,y_1\dotsm y_q)$ have the same number of points in each degree. Therefore $H_D=H_{\text{mon}(I_D)}=H_{p,q}$.

Case (2) $q<p$. Then (from \eqref{eq:ideal}),
the options for the initial monomial of $f$ are:
\[
\mon(f)=\begin{cases}y_1\dotsm y_q,& q < p-1,\\x_1\dotsm x_{p-1},& q=p-1.\end{cases}
\]
In each of these cases, we can compute the initial monomial ideal of $I_D$. In the first case,
$\mon(I_D)=(x_1\dotsm x_p,y_1\dotsm y_q)$. In the second case, $\mon(I_D)=(x_py_1\dotsm y_q,x_1\dotsm x_{p-1})$. In both cases, $H_D=H_{p,q}$. This completes the proof of the lemma.
\end{proof}

\begin{corollary}\label{cor:HSlemma}
In the settings of Lemma \ref{lem:HSlemma}, if there are $p'$, $q'$ such that $H_{p',q'}\geq H_{D}$ at $a\in \Sigma_{p,q}$ then $H_{p',q'}\geq H_{p,q}$. 
\end{corollary}
\begin{proof}
Without loss of generality we can assume that $p'\leq q'$. As in the proof of Lemma \ref{lem:HSlemma} we pass to the completion of the local ring at $a$ in $Y$. We also have that 
\[
 I_{D}=(x_1\dotsm x_p, x_py_1\dotsm y_r)+(f)\cap H.
\]
Recall that $r\geq q$ and $\ord(f)\geq q$. If $p>q$ then $\ord(I_D)\geq q$. Since $H_{p',q'}\geq H_D$ we must have $p',q'\geq q$ and then $H_{p',q'}\geq H_{p,q}$. If $p\leq q$ then $\ord(I_D)=p$. Since $H_{p',q'}\geq H_{D}$ we must have $\min(p',q')=p'\geq p=\min(p,q)$. Any element of $I_{D}$ of order $<q+1$ has initial monomial divisible by $x_1\dotsm x_p$, therefore the inequality $H_{p',q'}\geq H_{D}$ is not possible if $q'<q$. Hence $p'\geq p$ and $q'\geq q$, i.e. $H_{p',q'}\geq H_{p,q}$.
\end{proof}

\begin{remark}\label{note:HSlemma}
 Lemma \ref{lem:HSlemma} is at the core of our proof of Theorem \ref{thm:main}. The lemma describes the ideal of the support of $D$ at a point $a\in\Sigma_{p,q}$, under the following assumptions:
\begin{enumerate}
 \item $X$ is snc at $a$ and, after removing its last component, the resulting pair (with $D$) is semi-snc at $a$.
\item No component of $D$ at $a$ lies in the singular locus of $X$.
\item $H_{\Supp D,a}=H_{p,q}$.
\end{enumerate}
Under these assumptions we see that 
\[
(x_p=0)\cap(x_1\dotsm x_{p-1}=y_1\dotsm y_q=0)\subset(x_1\dotsm x_{p-1}=0)\cap(x_p=f=0)
\] 
(as in Example \ref{example2}), and also
\begin{equation*}
\begin{split}
 (x_1\dotsm x_{p-1}=y_1\dotsm y_q=0)\cap(x_p=f=0)\\=(x_p=x_1\dotsm x_{p-1}=y_1\dotsm y_q=0);
\end{split}
\end{equation*}
i.e., the intersection of $D^{p-1}:=(x_1\dotsm x_{p-1}=y_1\dotsm y_q=0)$ and $D_p:=(x_p=f=0)$ has only components of codimension $2$ in $X$.
\end{remark}
The previous statement is in fact true without the assumption $(1)$ of 
Remark \ref{note:HSlemma}. This stronger version will likely be useful 
to prove Theorem \ref{thm:main} without using an ordering of the components of $X$. The stronger lemma can be stated as follows.

\begin{lemma}\label{lem:strongHSlemma}
 Assume that $X$ is snc and no component of $D$ lies in the singular locus of $X$. Let $X_i$, $i=1,\ldots,n$, be the irreducible components of $X$ at $a$, and let $D_i$ be the divisorial part of $D|_{X_i}$. If $a\in X$ belongs to the stratum $\Sigma_{p,q}$ and $H_{\Supp D,a}=H_{p,q}$, then, for every $i,j$, the irreducible components of the intersection $D_i\cap D_j$ are all of codimension $2$ in $X$.
\end{lemma}

We plan to publish a proof of this lemma in a future paper.\medskip

\begin{proof}[\bf{Proof of Proposition \ref{lemmafactorssnc}}]
The assertion is trivial at a point in $X\setminus X_{m}$, so we assume that $a\in X_m$. 

At a semi-snc point $a$ of the pair $(X,D)$ the conditions are clearly satisfied. In fact, the ideal of $D$ is of the form $(x_1\dotsm x_p,y_1\dotsm y_q)$ in a system of coordinates for $Y$ at $a=0$ (recall that $D$ is reduced). We can then compute
 \[
 J_{a}=[(x_p,\,x_1\dotsm x_{p-1},\,y_1\dotsm y_q):(x_p,\,x_1\dotsm x_{p-1},\,y_1\dotsm y_q)]=\mathcal{O}_{Y,a}.
 \]
 
 Assume the conditions (1)--(3). By (1), there is system of coordinates $(x_1,\ldots, x_p,\, y_1, \ldots, y_q,\, z_1, \ldots, z_{n-p-q})$ for $Y$ at $a$, in which $X_m=(x_p=0)$ and $D$ is of the form
\[
D=(x_1\dotsm x_{p-1}=y_1\dotsm y_q=0)+(x_p=f=0).
\]
By Condition (2) and Lemma \ref{lem:HSlemma}, we can choose $f\in(x_1\dotsm x_{p-1},\,\allowbreak y_1\dotsm y_q,\, x_p)$ and, therefore, we can choose $f\in(x_1\dotsm x_{p-1},\,y_1\dotsm y_q)$. Write $f$ in the form $f=x_1\dotsm x_{p-1}g_1+y_1\dotsm y_q g_2$. Then
\begin{equation}\label{eq:Jg2unit}
\begin{aligned}
 J_{a}&=[(x_p,\,x_1\dotsm x_{p-1},\,f):(x_p,\,x_1\dotsm x_{p-1},\,y_1\dotsm y_q)]\\
   &=[(x_p,\,x_1\dotsm x_{p-1},\,y_1\dotsm y_q g_2):(x_p,\,x_1\dotsm x_{p-1},\,y_1\dotsm y_q)]\\
   &=(x_p,\,x_1\dotsm x_{p-1},\,g_2).
 \end{aligned}
\end{equation}
The condition $J_{a}=\mathcal{O}_{Y,a}$ means that $g_2$ is a unit. Then
\begin{align*}
D&=(x_1\dotsm x_{p-1}=y_1\dotsm y_q=0)+(x_p=f=0)\\
 &=(x_1\dotsm x_{p-1}=y_1\dotsm y_qg_2=0)+(x_p=f=0)\\
 &=(x_1\dotsm x_{p-1}=x_1\dotsm x_{p-1}g_1+y_1\dotsm y_qg_2=0)+(x_p=f=0)\\
 &=(x_1\dotsm x_p=f=0).
\end{align*}

By Lemma \ref{lem:HSlemma}, since $a\in\Sigma_{p,q}$, $\ord(f)=q$. It follows that 
$f|_{(x_p=0)}$ is a product $f_1\dotsm f_q$ of $q$ irreducible factors each of order one. For each $f_i$ set $I_i:=\{(j,k): f_i\in(x_j,y_k)|_{x_p=0},\ j\leq p-1,\ k\leq q\}$ then $f_i\in\cap_{(j,k)\in I_i}(x_j,y_k)|_{(x_p=0)}$, where the intersection is understood to be the whole local ring if $I_i$ is empty. Note that $\cup_i I_i=\{(j,k):\ j\leq p-1,\ k\leq q\}$,
since $f\in(x_1\dotsm x_{p-1},y_1\dotsm y_q)$. 

We will extend each $f_i$ to a regular function on $Y$ (still denoted $f_i$) preserving this condition, i.e. such that $f_i\in\cap_{(j,k)\in I_i}(x_j,y_k)$. In fact, $\cap_{(j,k)\in I_i}(x_j,y_k)|_{(x_p=0)}$ is generated by a finite set of monomials $\{m_r\}$ in the
$x_j|_{(x_p=0)}$ and $y_k|_{(x_p=0)}$. Then $f_i$ is a combination, $\sum m_ra_r$, of these monomials. So we can get an extension of $f_i$ as desired, using arbitrary 
extensions of the $a_r$ to regular functions on $Y$. This means we can assume that $f=f_1\dotsm f_q\in(x_1\dotsm x_{p-1},\,y_1\dotsm y_q)$ (using the extended $f_i$).

Since $f|_{(x_1=\dotsm=x_p=0)}=y_1\ldots y_q g_2$ where $g_2$ is a unit, it follows that
$f=y_1\ldots y_q g_2$ mod $(x_1,\ldots,x_p)$, where $g_2$ is a unit. Because $D=(x_1\dotsm x_{p},f)$, it remains only to check that $x_1,\ldots,x_p,f_1,\ldots,f_q$ are part of a system of coordinates. We can pass to the completion of the ring with respect to its maximal ideal, which we can identify with a ring of formal power series in variables 
including $x_1,\ldots,x_p,\,y_1,\ldots,y_q$. It is enough to prove that the images of the
$f_i$ and $x_i$ in $\hat{m}/\hat{m}^2$ are linearly independent, where $\hat{m}$ is the maximal ideal of the completion of the local ring $\mathcal{O}_{X,a}$. If we put $x_1=\dotsm=x_p=0$ in the power series representing each $f_i$ we get
\[
(f_1\dotsm f_q)|_{(x_1=\ldots=x_p)}=y_1\dotsm y_q.
\]
This means that, after a reordering the $f_i$, each $f_i|_{(x_1=\ldots=x_p)}\in(y_i)$, and the desired conclusion follows.
 \end{proof}

\section{Algorithm for the main theorem}\label{sec:maintheorem}
In this section we prove Theorem \ref{thm:main}. We divide the proof into several steps or subroutines each of which specify certain blowings-up.\medskip

\noindent
\textbf{Step 1:} \emph{Make $X$ snc.} This can be done simply by applying Theorem \ref{cor:corollarytheoremB} to $(X,0)$. The blowings-up involved preserve snc singularities of $X$ and therefore also preserve the semi-snc singularities of $(X,D)$. After Step $1$
we can  therefore assume that $X$ is everywhere snc.\medskip

\noindent
\textbf{Step 2:} \emph{Remove components of $D$ lying inside the singular locus of $X$.} Consider the union $Z$ of the supports of the components of $D$ lying in the singular locus of $X$. Blowings-up as needed can simply be given by the usual desingularization of $Z$, followed by blowing up the final strict transform.

The point is that, locally, there is a smooth ambient variety, with coordinates $(x_1,\ldots,x_p,\ldots,x_n)$ in which each component of $Z$ is of the form $(x_i=x_j=0)$, $i<j\leq p$. Let $C$ denote the set of irreducible components of intersections of arbitrary subsets of components of $Z$. Elements of $C$ are partially ordered by inclusion. Desingularization of $Z$ involves blowing up elements of $C$ starting with the smallest, until all components of $Z$ are separated. Then blowing up the final (smooth) strict transform removes all components of $Z$.

After Step 2 we can therefore assume that no component of $D$ lies in the 
singular locus of $X$.\medskip

\noindent
\textbf{Step 3:} \emph{Make $(X,D_\text{red})$ semi-snc.} (I.e., transform $(X,D)$ by the blowings-up needed to make $(X,D_\text{red})$ semi-snc.) The algorithm for Step $3$ is given following Step $4$ below. 

We can now therefore assume that $X$ is snc, $D$ has no components in the singular locus of $X$ and $(X,D_\text{red})$ is semi-snc.
\medskip

\noindent
\textbf{Step 4:} \emph{Make $(X,D)$ semi-snc.} A simple combinatorial argument for Step $4$ will be given in Section \ref{sec:fixingmultiplicities}. This finishes the algorithm.\medskip

\noindent
\textbf{Algorithm for Step 3:} The input is $(X,D)$, where $X$ is snc, $D$ is reduced and no component of $D$ lies in the singular locus of $X$. We will argue by induction on the number of components of $X$. It will be convenient to formulate the inductive assumption in terms of triples rather than pairs.

\begin{definition}\label{def:triples}
Consider a triple $(X,D,E)$, where $X$ is an algebraic variety, and $D$, $E$ are Weil divisors on $X$. Let $X_1$, \ldots, $X_m$ denote the irreducible components of $X$ with a given ordering. We use the notation of Definition \ref{def:pairs}.
Define
\begin{align*} 
E^i&:=E|_{X^i}+(X-X^i)|_{X^i},\\
(X,D,E)^i&:=(X^i, D^i,E^i),
\end{align*}
where $(X-X^i)|_{X^i}$ is viewed as a divisor on $X^i$.
\end{definition}

Recall Definitions \ref{def:transf},  \ref{def:triplessnc} and Remark \ref{rem:transf}.

\begin{definition}\label{sigmatriple}
Given $(X,D,E)$, we write 
$\Sigma_{p,q}=\Sigma_{p,q}(X,D,E)$ to denote $\Sigma_{p,q}(X,D)$ (so the strata 
$\Sigma_{p,q}$ depend on $X$ and $D$ but not on $E$). See Definition \ref{def:sigmapq}.\end{definition}
 
\begin{theorem}\label{thm:fortriples}
Assume that $X$ is snc, $D$ is a reduced Weil divisor on $X$ with no component in the singular locus of $X$, and $E$ is a Weil divisor on $X$ such that $(X,E)$ is semi-snc. Then there is a composite of blowings-up with smooth centers $f:X'\rightarrow X$, such that:
\begin{enumerate}
\item Each blowing-up is an isomorphism over the semi-snc points of its target triple.
\item The transform $(X',D',\tE)$ of the $(X,D,E)$ by $f$ is semi-snc.
\end{enumerate}
\end{theorem}

\begin{proof}
The proof is by induction on the number of components $m$ of $X$.

\emph{Case $m=1$.} Since $m=1$, then $(X,D+E)$ is semi-snc if and only if $(X,D+E)$ is snc. This case therefore follows from Theorem \ref{thm:theoremB} applied to $(X,D+E)$.\smallskip

\emph{General case.} The sequence of blowings-up will depend on the ordering of the components $X_i$ of $X$. We will use the notation of Definitions \ref{def:pairs}, \ref{def:triples}. Since $X$ is snc and no component of $D$ lies in the singular locus of $X$, it follows that every component of $D$ lies inside exactly one component of $X$.\medskip

\emph{By induction, we can assume} that $(X^{m-1},D^{m-1},E^{m-1})$ is semi-snc. We want to make $(X^m,D^m,E^m)$ semi-snc. For this purpose, we only have to remove the unwanted singularities from the last component $X_m$ of $X=X^m$. 

Recall that $X$ is partitioned by the sets $\Sigma_{p,q}=\Sigma_{p,q}(X,D)$. Clearly for all $p$ and $q$, the closure 
$\overline{\Sigma}_{p,q}$ of $\Sigma_{p,q}$ has the property 
\[
\overline{\Sigma}_{p,q}\subset\bigcup_{p'\geq p,\,q'\geq q}\Sigma_{p',q'}.
\]
We will construct sequences of blowings-up $X'\rightarrow X$ such that $X'$ is semi-snc on certain strata $\Sigma_{p,q}(X',D')$, and then iterate the process. The following definitions are convenient to
describe the process precisely.

\begin{definitions}\label{def:monotone}
Consider the partial order on $\mathbb{N}^2$ induced by the \red{order on the set $\{\Sigma_{p,q}\}$, see Definition }\ref{def:orderSigmapq}. For $I\subset\mathbb{N}^2$, define the \emph{monotone closure} 
$\overline{I}$ of $I$ as
$\overline{I}:=\{x\in\mathbb{N}^2:\ \exists y\in I,\ x\geq y\}$. We say that $I\subset\mathbb{N}^2$ is \emph{monotone} if $\overline{I}=I$. The set of monotone subsets of $\mathbb{N}^2$ is partially ordered by inclusion, and has the property that any increasing sequence stabilizes. Given a monotone $I$  and a pair $(X,D)$, set 
\[
\Sigma_I(X,D)=\bigcup_{(p,q)\in I}\Sigma_{p,q}(X,D).
\]
\end{definitions}

Then $\Sigma_I(X,D)$ is closed. In fact, if $I$ is monotone then $\Sigma_{I}(X,D)=\bigcup_{(p,q)\in I}\overline{\Sigma}_{p,q}$. 

\begin{definition}\label{def:K}
Given $(X,D)$ and monotone $I$, let $K(X,D,I)$ denote the set of maximal elements of 
$\{(p,q)\in\mathbb{N}^2\setminus I:\ \Sigma_{p,q}(X,D)\neq\emptyset\}$. Also set
$K(X,D):=K(X,D,\emptyset)$. \red{Note that $K(X,D,I)$ consists only of incomparable pairs $(p,q)$ and that it does not simultaneously contain strata $\Sigma_{p,q}$ with $p\geq 3$, $p= 2$ and $p=1$.}
\end{definition}\medskip

\noindent\textbf{Case A:} \emph{We first deal with the case in which $K(X,D)$ contains strata $\Sigma_{p,q}$ with $p\geq 3$.} We can apply Proposition \ref{prop:pgreaterthan3} to reduce to the case in which $(X,D,E)$ is semi-snc at every point lying in at least $3$ components of $X$.\medskip

\noindent\textbf{Case B:} \emph{Assume that $(X,D,E)$ is semi-snc at every point lying in at least $3$ components of $X$.} Let $I:=\{(p,q)\in\mathbb{N}^2:\ p\geq3\}$ and $U$ the complement of $\Sigma_{I_k}(X,D)$. Assume that $K(U,D|_U)$ contains a stratum $\Sigma_{2,q}$, for some $q$. In particular this means that $K(U,D|_U)$ doesn't contain any stratum $\Sigma_{1,q}$. We can apply Proposition \ref{propositioncasepequal2} to $(X,D,E)|_{U}$ to reduce to the case in which $(X,D,E)$ is semi-snc at every point lying in at least $2$ components. Observe that the centers involved never intersect a stratum $\Sigma_{p,q}$ with $p\neq2$.
\medskip

\noindent\textbf{Case C:} \emph{Finally, assume that $(X,D,E)$ is semi-snc at every point in $\Sigma_{p,q}$ for $p\geq2$.} Recall that if $X$ has only one component (and is therefore smooth), then semi-snc is the same as snc. Hence this case follows from Theorem \ref{thm:theoremB} applied to the pair $(X^m,D^m+E^m)|_U$, where $U$ is the complement of the union of all $\Sigma_{p,q}$ with $p\geq2$.
\end{proof}

\begin{remark}\label{rem:closedcenters}
The centers of blowing up used in Proposition \ref{propositioncasepequal2} (Case B) and also in Theorem \ref{thm:theoremB} (Case C) are closed in $U$ and contain only non-semi-snc points. Since $(X,D,E)$ is semi-snc 
on $\Sigma_{I_k}(W_{j'_k},F_{j'_k})$, and therefore in a neighborhood of the latter, we  see that these centers are also closed in $W_{j'_k}$.
\end{remark}


\section{The case of more than 2 components}\label{sec:morethantwocomponents}
In this section, we show how to remove the unwanted singularities in the strata 
$\Sigma_{p,q}(X,D)$, with $p\geq3$.

Throughout the section, $(X,D,E)$ denotes a triple as in Definition \ref{def:triples},
and we use the notation of the latter. As in Theorem \ref{thm:fortriples}, we assume
that $X$ is snc, $D$ is reduced and has no component in the singular locus of $X$, and 
$(X,E)$ is semi-snc. We consider $K(X,D)$ as in Definition \ref{def:K}.

\begin{proposition}\label{prop:pgreaterthan3}
With the hypothesis of Theorem \ref{thm:fortriples}, assume that $(X,D,E)$ is such that $K(X,D)$ contains a stratum $\Sigma_{p,q}$ with $p\geq 3$. Then, there is a composition of admissible blowings-up\: $X'\rightarrow X$ such that $(X',D',\tE)$ is semi-snc at every point lying in at least $3$ components of $X$.
\end{proposition}

\begin{proof}
We start with the variety $W_0:=X$ and the divisors $F_0:=D$, $G_0=E$, and we define $I_0$ as the monotone closure of
\[
\{\text{maximal elements of }\{(p,q)\in\mathbb{N}^2:\ \Sigma_{p,q}(W_0,F_0)\neq\emptyset\}\}.
\]
Put $j_0 = 0$.
Inductively, for $k\geq0$, we will construct admissible blowings-up 
\begin{equation}\label{eq:2steps}
W_{j_k}\leftarrow\dotsm\leftarrow W_{j'_{k}}\leftarrow\dotsm\leftarrow W_{j_{k+1}}
\end{equation}
such that, if $(W_{j_{k+1}},F_{j_{k+1}}, G_{j_{k+1}})$ denotes the transform of the triple 
$(W_{j_k},F_{j_k}, G_{j_k})$, then $(W_{j_{k+1}},F_{j_{k+1}})$ semi-snc on $\Sigma_{I_k}(W_{j_{k+1}},F_{j_{k+1}})$. Then we define 
\[
I_{k+1}:=\overline{I_k\cup K(W_{j_{k+1}},F_{j_{k+1}},I_k)}. 
\]
We have $I_{k+1}\supset I_k$, with equality only if $\Sigma_{I_k}(W_{j_k},F_{j_k})=W_{j_k}$.

In this way we define a sequence $I_0\subset I_1\subset\ldots$. Since this sequence stabilizes, there is $t$ such that $\Sigma_{I_t}(W_{j_t},F_{j_t})=W_{t}$. By construction, $W_{j_t}$ is semi-snc on $\Sigma_{I_t}(W_{j_t},F_{j_t})$, so that
$(W_{j_t},F_{j_t})$ is everywhere semi-snc.
\medskip

The blowing-up sequence \eqref{eq:2steps} will be described in two steps.
The first provides a sequence of admissible blowings-up $W_{j_k}\leftarrow\ldots\leftarrow W_{j'_k}$ for the purpose of making the Hilbert-Samuel function equal to $H_{p,q}$ on 
$\Sigma_{p,q}$, for each $(p,q)\in K(W_{j'_k},F_{j'_k})$. The second step provides a sequence of admissible blowings-up $W_{j'_k}\leftarrow\ldots\leftarrow W_{j_{k+1}}$ that finally removes the non-semi-snc points from the $\Sigma_{p,q}$, where $(p,q)\in K(W_{j_{k+1}},F_{j_{k+1}})$. 
\medskip

\noindent
\textbf{Step 1:} We can assume that, locally, $X+E$ is embedded as an snc hypersurface in a smooth variety $Z$. We consider the embedded desingularization algorithm applied to $\Supp D$ with the divisor $X+E$ in $Z$. We will blow up certain components of the centers of blowing up involved. These centers are the maximum loci of the desingularization invariant, which decreases after each blowing-up. Our purpose is to decrease the Hilbert-Samuel function, which is the first entry of the invariant. During the desingularization process, some components of $X+E$ may be moved away from 
$\Supp D$ before $\Supp D$ becomes smooth. We will only use centers from the desingularization algorithm that contain no semi-snc points. By assumption, all non-semi-snc points lie in $X_m$, so that all centers we will consider are inside $D_m$. Therefore $X_m$ (which is a component of $X+E$) is not moved away before $D_m$ 
becomes smooth. 

We are interested in the maximum locus of the invariant on the complement $U_k$ of $\Sigma_{I_k}(W_{j_k},F_{j_k})$ in $W_{j_k}$. The corresponding blowings-up are used to decrease the maximal values of the Hilbert-Samuel function. 

\begin{lemma}\label{lem:centersinsigmapq}
Let $C$ be an irreducible smooth subvariety of $\Supp D$. Assume that the Hilbert-Samuel function equals $H_{p,q}$ (for given $p$, $q$) at every point of $C$. If 
$C \cap \Sigma_{p,q} \neq \emptyset$, then $C\subset\Sigma_{p,q}$.
\end{lemma}

\begin{proof}
Let $a\in C \cap \Sigma_{p,q}$. Since the Hilbert-Samuel function of $\Supp D$ is constant on $C$, then $a$ has a neighborhood $U\subset C$, each point of which lies in
precisely those components of $D$ at $a$.  Therefore, $U\subset\Sigma_{p,q}$. Since the closure of $\Sigma_{p,q}$ lies in the union of the $\Sigma_{p',q'}$ with $p'\geq p$,  $q'\geq q$, any $b \in C \setminus U$ belongs to $\Sigma_{p',q'}$, for some $p'\geq p$, $q'\geq q$. Thus $H_{\Supp D,b} = H_{p,q} \leq H_{p',q'}$. But, by Lemma \ref{lem:HSlemma}, the Hilbert-Samuel function cannot be $< H_{p',q'}$ on 
$\Sigma_{p',q'}$ . Therefore $b \in \Sigma_{p,q}$.
\end{proof}

We write the maximum locus of the invariant in $U_k$ as a disjoint union $A\cup B$ in the following way: $A$ is the union of those components of the maximum locus containing no semi-snc points, and $B$ is the union of the remaining components. Thus $B$ is
the union of those components of the maximum locus of the invariant with generic point semi-snc. Each component of $B$ has Hilbert-Samuel function $H_{p,q}$, for some $p,q$, and lies in the corresponding $\Sigma_{p,q}$ by Lemma \ref{lem:centersinsigmapq}.
On the other hand, any component $C$ of the maximum locus of the invariant where either the invariant does not begin with $H_{p,q}$, for some $p, q$, or the invariant begins with some $H_{p,q}$ but no point of $C$ belongs to $\Sigma_{p,q}$, is a component 
of $A$. 

Both $A$ and $B$ are closed in the open set $U_k\subset W_{j_k}$. $B$ is not necessarily closed in $W_{j_k}$. But all points in the complement of $U_k$ are semi-snc,
and the semi-snc points are open. Since no points of $A$ are semi-snc, $A$ has no limit points in the complement of $U_k$. Thus $A$ is closed in $W_{j_k}$.

We blow up with center $A$. Then the invariant decreases in the preimage 
of $A$. Recall that $A$ and $B$ depend on $(X,D)$. We use the same notation $A$ and $B$ to 
denote the sets with the same meaning as above, after blowing up. So we can continue to blow up until $A=\emptyset$. Say we are now in year $j_k'$.

\begin{claim}\label{claim:claimAempty}
If $(p,q)\in K(W_{j'_{k}},F_{j'_{k}})$ (so that $A=\emptyset$), then the 
Hilbert-Samuel function equals $H_{p,q}$ at every point of $\Sigma_{p,q}$. 
\end{claim}
\begin{proof}
Let $a\in\Sigma_{p,q}$, where $(p,q)\in K(W_{j'_k},F_{j'_k})$. Assume that the Hilbert-Samuel function $H$ at $a$ is not equal to $H_{p,q}$. Recall that every point of $B$ has Hilbert-Samuel function of the form $H_{p',q'}$ for some $p', q'$,  and belongs to 
$\Sigma_{p',q'}$. Therefore $a\notin B$, so the invariant at $a$ is not maximal. Thus there is $b\in B$ where the Hilbert-Samuel function is $H_{p',q'}>H$ for some $p', q'$ and $b\in\Sigma_{p',q'}$. By Corollary \ref{cor:HSlemma}, $H_{p',q'}>H_{p,q}$. This means that $\Sigma_{p',q'}>\Sigma_{p,q}$. Since $(p,q)\in K(W_{j'_k},F_{j'_k})$ then $(p',q')\in I_{k}$. We have reached a contradiction because $b\in B$ and $B$ lies in the complement of $\Sigma_{I_k}(W_{j'_k},F_{j'_k})$.
\end{proof} 

The claim \ref{claim:claimAempty} shows that when $A=\emptyset$ we have achieved the goal of Step $1$, i.e., the Hilbert-Samuel function equals $H_{p,q}$ at every point of
$\Sigma_{p,q}$, where $(p,q)\in K(W_{j'_k},F_{j'_k})$.
\medskip

\noindent
\textbf{Step 2:} We now describe blowings-up that eliminate non-semi-snc points
from the strata $\Sigma_{p,q}$, with $(p,q)\in K(W_{j'_k},F_{j'_k})$. Note this does
not mean that all the points in the preimage of these strata will be semi-snc. Only the 
points of the strata $\Sigma_{p,q}$, for the transformed $(X,D,E)$, for $(p,q)\in K(W_{j'_k},F_{j'_k})$, will be made semi-snc. The remaining points of the preimages will belong to new strata $\Sigma_{p',q'}$, where $p'<p$ or $q'<q$ and therefore will be treated in further iterations of Steps 1 and 2. 

\emph{We are assuming that $K(W_{j'_k},F_{j'_k})$ contains some stratum $\Sigma_{p,q}$ with $p\geq3$.} Hence, by Definition \ref{def:orderSigmapq}, all strata in $K(W_{j'_k},F_{j'_k})$ is of the form $\Sigma_{p,q}$ with $p\geq3$. Therefore this case follows from Proposition \ref{pro:pgreaterthanthree} below 
applied to $(X,D,E)|_{U}$, where $U$ is the complement of $\Sigma_{I_k}(W_{j'_k},F_{j'_k})$ in $W_{j'_k}$. Observe that the center of the blowing-up involved never intersects a stratum $\Sigma_{p,q}$ with $p\leq2$. 
\end{proof}

The following lemmas are needed to state Proposition \ref{pro:pgreaterthanthree}.

\begin{lemma}\label{lem:center}
Assume that $(X^{m-1},D^{m-1},E^{m-1})$ is semi-snc and let $(p,q)\in K(X,D)$.
Define
\begin{equation}\label{eq:tC}
C_{p,q} := X_m\cap\Sigma_{p-1,q}(X^{m-1},D^{m-1}).
\end{equation}
Then:
\begin{enumerate}
\item
$C_{p,q}$ is smooth;
\item
$\Sigma_{p,q}(X,D) \subset C_{p,q} \subset \bigcup_{q'\leq q}\Sigma_{p,q'}(X,D)$.
\end{enumerate}
\end{lemma}

\begin{lemma}\label{lem:pgreaterthanthree}
Assume that $(X^{m-1},D^{m-1},E^{m-1})$ is semi-snc and let $(p,q)\in K(X,D)$.
Assume that $p\geq 3$ and that the Hilbert-Samuel function equals
$H_{p,q}$, at every point of $\Sigma_{p,q} = \Sigma_{p,q}(X,D)$. Then:
\begin{enumerate}
\item 
Every irreducible component of $C_{p,q}$
which contains a non-semi-snc point of $\Sigma_{p,q}$ consists entirely of
non-semi-snc points.
\item
Every irreducible component of $\Sigma_{p,q}$ consists entirely
either of semi-snc points or non-semi-snc points.
\end{enumerate}
\end{lemma}

\begin{definition}\label{def:center}
Assume that $(X^{m-1},D^{m-1},E^{m-1})$ is semi-snc and that, for all 
$(p,q)\in K(X,D)$, where $p \geq 3$, the Hilbert-Samuel function equals
$H_{p,q}$, at every point of $\Sigma_{p,q}$.
Let $C$ denote the union over all $(p,q)\in K(X,D)$, $p \geq 3$, of the union of all components 
of $C_{p,q}$ which contain non-semi-snc points of $\Sigma_{p,q}$.
\end{definition}

\begin{proposition}\label{pro:pgreaterthanthree}
Under the assumptions of Definition \ref{def:center}, let $\s: X' \to X$ denote
the blowing-up with center $C$ defined above. Then:
\begin{enumerate}
\item\label{pointtwo} The transform $(X',D',\tE)$ of $(X,D,E)$ is semi-snc on 
the strata $\Sigma_{p,q}(X',D')$, for all $(p,q)\in K(X,D)$ with $p \geq 3$.
\smallskip
\item\label{pointthree} Let $a \in \Sigma_{p,q}$, where $(p,q)\in K(X,D)$ and $p \geq 3$. If $a\in C$ and $a' \in \s^{-1}(a)$, then $a'\in \Sigma_{p',q'}(X',D')$, where $p'\leq p$, $q'\leq q$, and at least one of these inequalities is strict.
\end{enumerate}
\end{proposition}

\begin{proof}[Proof of Lemma \ref{lem:center}]
This is immediate from the definitions of 
$\Sigma_{p,q} = \Sigma_{p,q}(X,D)$, $K(X,D)$ and $C_{p,q}$.
\end{proof}

\begin{proof}[Proof of Lemma \ref{lem:pgreaterthanthree}]
Let $a \in \Sigma_{p,q}$ be a non-semi-snc
point, and let $S$ be the irreducible component of $\Sigma_{p,q}$ containing $a$.  Let $C_0$ denote the component of $C_{p,q}$ containing
$S$. We will prove that all points of $C_0$ are non-semi-snc, as required for (1). In particular, all points in $S$ are non-semi-snc and (2) follows.

By Lemma \ref{lem:HSlemma}, $X$ is embedded locally at $a$ in a smooth
variety $Y$ with a system of coordinates $x_1,\ldots,x_p,y_1,\ldots,y_q, z_1,\ldots,z_{n-p-q}$ in a neighborhood $U$ of $a=0$, in which we can write:
\begin{align*}
X_m&=(x_p=0),\\
X&=(x_1\dotsm x_p=0),\\
D&=D^{m-1}+D_m,
\end{align*}
where 
\begin{align*}
D^{m-1}&:=(x_1\dotsm x_{p-1}=y_1\dotsm y_q=0),\\
D_m&:=(x_p=x_1\dotsm x_{p-1}g_1+y_1\dotsm y_qg_2=0).
\end{align*}

Since $(X,D,E)$ is not semi-snc at $a$ then $g_2$ is not a unit (see 
Lemma \ref{lemmafactorssnc}(3) and \eqref{eq:Jg2unit}). In fact, by 
Lemma \ref{lem:compJ} following, the ideal $J(X,D)$ (see Definition 
\ref{definitionJ}) is given at $a$ by $(x_p,x_1\dotsm x_{p-1},g_2)$; 
the latter coincides
with the local ring of $Y$ at $a$ if and only if $g_2$ is a unit. 
In the given coordinates,
\begin{equation}\label{eq:centercomp}
C_0=(x_1=\ldots=x_p=y_1=\ldots=y_q=0).
\end{equation}

To show that all the points in $C_0$ are non-semi-snc, it is enough to show that $g_2$ is in the ideal $(x_1,\ldots,x_p,y_1,\ldots,y_q)$. In fact, the latter implies that $g_2$ is not 
a unit, and therefore that $J(X,D)=(x_p,x_1\dotsm x_{p-1},g_2)$ is a proper ideal at
every point of $C_0 \cap U$. Since $C_0$ is irreducible, $C_0 \cap U$ is dense in
$C_0$. But the set of semi-snc is open, so it follows that all points in $C_0$ are non-semi-snc. 

Proposition \ref{propositionreductionofp} below shows that if $g_2$ is not a unit, then $g_2\in(x_1,\ldots,x_p,y_1,\allowbreak\ldots,y_q)$, concluding the proof of Lemma
\ref{lem:pgreaterthanthree}.
\end{proof}

\begin{lemma}\label{lem:compJ}
Let $R$ denote a regular local ring, and suppose that $x_1,\ldots,x_p,\allowbreak y_1,\ldots,y_q,z_1,\ldots,z_{n-p-q}$ is a regular system of parameters of $R$. 
If $g_2\in R$, then the quotient ideal
\begin{equation*}
\begin{split}
[(x_p,x_1\dotsm x_{p-1},y_1\dotsm y_q g_2):(x_p,x_1\dotsm x_{p-1}, y_1\dotsm y_q)]\\= (x_p,x_1\dotsm x_{p-1}, g_2).
\end{split}
\end{equation*}
\end{lemma}
\begin{proof}
We have
\begin{align*}
\begin{split}
[(x_p,x_1\dotsm x_{p-1}&,y_1\dotsm y_q g_2):(x_p,x_1\dotsm x_{p-1}, y_1\dotsm y_q)]\\&={[(x_p,x_1\dotsm x_{p-1},y_1\dotsm y_q g_2):(y_1\dotsm y_q)]}\\
 &=\frac{1}{y_1\dotsm y_q}\cdot\left[(x_p,x_1\dotsm x_{p-1},y_1\dotsm y_q g_2)\cap (y_1\dotsm y_q)\right].
\end{split}
\end{align*}
Of course,
\begin{equation*}
(x_p,x_1\dotsm x_{p-1},y_1\dotsm y_q g_2)\cap (y_1\dotsm y_q)
 \supset y_1\dotsm y_q\cdot(x_p,x_1\dotsm x_{p-1},g_2).
\end{equation*}

To prove the reverse inclusion, assume that $h$ belongs to the left hand side. 
Then we can write $h=x_pa+x_1\dotsm x_{p-1}b+y_1\dotsm y_qg_2c$. Since 
$h\in(y_1\dotsm y_q)$, then $x_pa+x_1\dotsm x_{p-1}b\in(y_1\dotsm y_q)$. This 
implies that $x_1\dotsm x_{p-1}b\in(x_p,y_1\dotsm y_q)$. Since 
$(x_p,y_1\dotsm y_q)$ is an intersection of primes, none of which contains 
$x_k$, $k=1,\ldots, p-1$, then $b\in(x_p,y_1\dotsm y_q)$. Therefore, we can 
write $h=x_pa'+x_1\dotsm x_{p-1}b'+\allowbreak y_1\dotsm y_qg_2c$, where 
$b' \in(y_1\dotsm y_q)$. This implies that $x_pa' \in(y_1\dotsm y_q)$, 
and therefore that $a' \in(y_1\dotsm y_q)$. Hence 
$h\in y_1\dotsm y_q\cdot(x_p,x_1\dotsm x_{p-1},g_2)$. This gives the reverse 
inclusion required to complete the proof.
\end{proof}

\begin{proof}[Proof of Proposition \ref{pro:pgreaterthanthree}]
With reference to the proof of Lemma \ref{lem:pgreaterthanthree}, 
it is clear from \eqref{eq:centercomp} that blowing up $C_0$, either $p$ or $q$ decreases in the preimage. This implies (\ref{pointthree}) in the proposition. It also implies that, after the blowing-up $\s$ of $C$, all points in the preimage of $\Sigma_{p,q}(X,D)$ which belong to $\Sigma_{p,q}(X',D')$ are semi-snc. This establishes (\ref{pointtwo}).
\end{proof}

\begin{proposition}\label{propositionreductionofp}
Let $f$ denote an element of a regular local ring. Assume that $f$ has $q$ irreducible factors, each of order $1$, that $f\in(x_1\dotsm x_{p-1},y_1\dotsm y_q)$, where 
$p\geq3$ and the $x_i$, $y_i$ form part of a regular system of parameters, and that $f=x_1\dotsm x_{p-1}g_1+y_1\dotsm y_q g_2$, where $g_2$ is not a unit. Then
$g_2\in(x_1,\ldots,x_{p-1},y_1,\ldots,y_q)$.
 \end{proposition}

 \begin{remark}
The condition $p\geq3$ is crucial, as can be seen from Example \ref{example2}. In
the latter, we have $D=(x_1=y_1=0)+(x_2=x_1+y_1z=0)$, so that $f = x_1+y_1z$ and
$g_2=z\notin(x_1,x_2,y_1)$.
 \end{remark}
 
To prove Proposition \ref{propositionreductionofp} we will use the following lemma.

 \begin{lemma}\label{lemmareductionofp}
Let $p\geq 3$ and $s\geq 0$ be integers. Consider
\begin{equation}\label{eq:givenfactors}
f=(x_1m_1+a_1)\dotsm(x_{p-1}m_{p-1}+a_{p-1})(y_{r_1}n_1+b_1)\dotsm(y_{r_s}n_s+b_s)g,
\end{equation}
where the $x_i$, $y_i$, $a_i$, $b_i$, $m_i$, $n_i$ and $g$ are elements of a regular local ring with $x_1,\ldots,x_{p-1}$, $y_1,\ldots,y_q$ part of a regular system of parameters, 
and $1 \leq r_1 < \dotsm < r_s \leq q$.
Assume that, for every $i=1,\ldots,p-1$ and $j=1,\ldots,q$,
\begin{equation}\label{eq:star}
\text{\parbox{23em}{
if $a_i\notin(y_j)$, then $y_j=y_{r_k}$ and $b_{k}\in(x_i)$, for some $k$.
 }}
\end{equation}
Then, after expanding the right hand side of \eqref{eq:givenfactors}, all the monomials 
(in the elements above) appearing in the expression are in either the ideal $(x_1\dotsm x_{p-1})$ or the ideal 
 $(y_1\dotsm y_q)\cdot(x_1,\ldots,x_{p-1},y_1,\ldots,y_q).$
 \end{lemma}
 
\begin{remark}
The conclusion of the lemma implies that $f$ can be written as $x_1\dotsm x_{p-1}g_1+y_1\dotsm y_qg_2$ with $g_2\in\allowbreak(x_1,\ldots,x_{p-1},y_1,\ldots,y_q)$. This is precisely what we need for Proposition \ref{propositionreductionofp}.
\end{remark}

\begin{proof}[Proof of Lemma \ref{lemmareductionofp}]
First consider $s=0$. Then \eqref{eq:star} implies that each $a_i$ is in the ideal $(y_1\dotsm y_q)$. The expansion of
\[
(x_1m_1+a_1)\dotsm(x_{p-1}m_{p-1}+a_{p-1}),
\]
includes the monomial $x_1\dotsm x_{p-1}m_1\ldots m_{p-1}$, which belongs to the ideal $(x_1\dotsm x_{p-1})$. Each of the remaining monomials is a multiple of some 
$x_ia_j$ or of some $a_ia_j$, and therefore belongs to $(y_1\dotsm y_q)\cdot(x_1,\ldots,x_{p-1},y_1,\ldots,y_q)$.
 
By induction, assume the lemma for $p,s-1$, where $s\geq 1$. Consider $f$ as in  the lemma (for $p,s$). Then $f/(y_{r_s}n_s+b_s)$ satisfies the hypothesis of the lemma (with $s-1$) when $y_{r_s}$ is deleted from the given elements of the ring. (Note that the lemma also depends on $q$. Here we are using it for $s-1$ and $q-1$.) Then, by induction, all the terms appearing after expanding $f/(y_{r_s}n_s+b_s)$ are either in the ideal $(x_1\dotsm x_{p-1})$ or in the ideal
\begin{equation}\label{eq:factor}
\left(\frac{y_1\dotsm y_q}{y_{r_s}}\right)\cdot(x_1,\ldots,x_{p-1},y_1,\ldots,y_q).
\end{equation}

Assume there is a term $\xi$ appearing after expanding \eqref{eq:givenfactors}
which is not in $(x_1\dotsm x_{p-1})$. Then there is $x_k$ such that $\xi\notin(x_k)$. Then $\xi$ is divisible by $a_k$, according to \eqref{eq:givenfactors}, and $\xi$ belongs to
the ideal \eqref{eq:factor}.

If $a_k\in(y_{r_s})$, we are done. By \eqref{eq:givenfactors}, $\xi$ is a multiple either of $y_{r_s}n_s$ or $b_s$. If $a_k\notin(y_{r_s})$, and if we assume that $\xi$ was obtained by multiplying by $b_s$ rather than by $y_{r_s}n_s$, then $\xi$ is divisible by $x_k$, which is a contradiction.
 \end{proof}
 
 \begin{proof}[Proof of Proposition \ref{propositionreductionofp}]
 To prove this proposition it is enough to show that $f$ can be written as a product as in the previous lemma.
To begin with, $f=h_1\dotsm h_q\in(x_1\dotsm x_{p-1},y_1\dotsm y_q)=\cap(x_i,y_j)$. Since each $(x_i,y_j)$ is prime, it follows that, for each $i=1,\ldots,p-1$ and 
$j=1,\ldots,q$, there is a $k$ such that $h_k\in(x_i,y_j)$. If there is a unit $u$ such that $h_k=y_ju+a$, where $\ord(a)\geq2$, then we say that $h_k$ is \emph{associated to $y_j$}; otherwise we say that $h_k$ is \emph{associated to $x_i$}. There may be $h_k$ that belong to no $(x_i,y_j)$ and are, therefore, not associated to any $x_i$ or $y_j$.
 
 By definition, any $h = h_k$ cannot be associated to some $x_i$ and $y_j$ at the same time. Let us prove that $h$ can be associated to at most one $x_i$. Assume that $h$ is associated to $x_{i_1}$ and $x_{i_2}$, where $i_1\neq i_2$. Then $h\in (x_{i_1},y_{j_1})\cap(x_{i_2},y_{j_2})$, for some $j_1$ and $j_2$. If $j_1\neq j_2$, then $h$ cannot be of order $1$, since $(x_{i_1},y_{j_1})\cap(x_{i_2},y_{j_2})$ only contains elements of order $\geq 2$. If $j_1=j_2$ then $(x_{i_1},y_{j_1})\cap(x_{i_2},y_{j_2})=(x_{i_1}x_{i_2},y_{j_1})$, but this would mean that $h$ is associated to $y_{j_1}$, and therefore not to $x_{i_1}$ or $x_{i_2}$.

An analogous argument shows that an $h$ cannot be associated to two different $y_j$. Therefore, the collection of $h_k$ is partitioned into those associated to a unique $x_i$, those associated to a unique $y_j$ and those associated to neither some $x_i$ nor 
some $y_j$.
 
We now show that, for each $i=1,\ldots,p-1$, there exists $h= h_k$ associated to $x_i$.  Assume there is an $x_i$ (say $x_1$) with no associated $h$. For each $j=1,\ldots,$,
there exists $k_j$ such that $h_{k_j}\in(x_1,y_j)$. Then $h_{k_j}$ is associated to 
$y_j$. It follows that each $k_j$ corresponds to a unique $j$. Thus, after reordering the 
$h_k$, we have $h_i$ is associated to $y_i$, for each $i=1,\ldots,q$. This means that $h_i=y_iu_i+a_i$, where $u_i$ is a unit and $\ord\,a_i \geq2$. This contradicts the assumption that $g_2$ is not a unit. Therefore, for each $i=1,\ldots,p-1$, there exists
$h_k$ associated to $x_i$.
 
We take the product of all members of each set in the partition above. The product of all $h_k$ associated to $x_i$ can be written as $x_im_i+a_i$, and it satisfies the property that
\begin{equation}\label{propertyone}
x_im_i+a_i\notin(x_\alpha,y_\beta)\,\,\text{ unless }\alpha=i.
\end{equation}
In fact, if $x_im_i+a_i\in(x_\alpha,y_\beta)$ then there exists $h=h_k$ associated to $x_i$ such that $h\in(x_\alpha,y_\beta)$. But then $h$ is associated to either $y_\beta$ or $x_\alpha$, which contradicts the condition that $h$ is associated to $x_i$, where $i\neq\alpha$. 

In the same way, write the product of all $h_k$ associated to $y_{r_i}$ as $y_{r_i}m_i+b_i$. Then
\begin{equation}\label{propertytwo}
y_{r_i}m_i+b_i\notin(x_\alpha,y_\beta)\,\,\text{ unless }\beta=r_i.
\end{equation}
Also write the product of all $h_k$ not associated to any $x_i$ or $y_j$ as $g$. We get the expression
 \begin{equation}\label{eq:expr}
f=(x_1m_1+a_1)\dotsm(x_{p-1}m_{p-1}+a_{p-1})(y_{r_1}n_1+b_1)\dotsm(y_{r_s}n_s+b_s)g,
\end{equation}
but \eqref{eq:expr} does not \emph{a priori} satisfy the hypotheses of Lemma \ref{lemmareductionofp}.
 
We will use the properties \eqref{propertyone} and \eqref{propertytwo} above to modify the elements $m_\cdot$, $a_\cdot$, $n_\cdot$ and $b_\cdot$ in \eqref{eq:expr} to get the hypotheses of the lemma.

We will check whether \eqref{eq:star} is satisfied, for all $i=1,\ldots,p-1$ and 
$j=1,\ldots,q$. Order the pairs $(i,j)$ reverse-lexicographically (or, in fact, in any way). 
Given $(i,j)$, assume, by induction, that
\eqref{eq:star} is satisfied for all $(i',j')<(i,j)$. Suppose that \eqref{eq:star} is not satisfied
for $(i,j)$. Then we will modify $m_\cdot$, $a_\cdot$, $b_\cdot$ and $n_\cdot$ so that \eqref{eq:star} will be satisfied for all $(i',j')\leq(i,j)$. We consider the following cases.\medskip

Case (1): $j\neq r_k$, for any $k$. Then, if $a_i\in(y_j)$, there is nothing to do.
If $a_i\notin(y_j)$, we can modify $a_i$ and
$m_i$ so that the new $a_i$ will satisfy $a_i\in(y_j)$, and \eqref{eq:star} will still be satisfied for $(i',j')<(i,j)$: Since $f\in(x_i,y_j)$ and, for every $k$, $y_j\neq y_{r_k}$, then $a_i\in(x_i,y_j)$. Write $a_i=ya$, where $y$ is a monomial in the $y_\ell$ and $a$ is divisible by no $y_\ell$. Then $a\in(x_i,y_j)$ and we can write $a=x_ig_1+y_jg_2$, $x_im_i+a_i=x_i(m_i+yg_1)+yy_jg_2$. Relabel $m_i+yg_1$ and $y_jyg_2$ as our new $m_i$ and $a_i$, respectively.  Then $a_i\in(y_j)$, and clearly \eqref{eq:star} is still satisfied for $(i',j')<(i,j)$.\medskip

Case (2): $j=r_k$, for some $k$. Since $f\in(x_i,y_j)$, then $a_ib_{k}\in(x_i,y_j)$. Since $(x_i,y_j)$ is prime, either $a_i\in(x_i,y_j)$ (in which case we proceed as before), or  $b_{k}\in(x_i,y_j)$. Consider the latter case. If $b_k\in(x_i)$, there is nothing to do. Assume $b_k\notin(x_i)$. Write $b_k=xb$, where $x$ is a monomial in the $x_\ell$ and $b$ is divisible by no $x_\ell$. Then $b\in(x_i,y_j)$. Thus we can write $b=x_ig_1+y_jg_2$ and $y_jm_{k}+b_{k}=y_j(m_{k}+xg_2)+x_ixg_1$. Relabel $m_{k}+xg_2$ and $x_ixg_1$ as our new $n_{k}$ and $b_{k}$, respectively. Then $b_{k}\in(x_i)$, and \eqref{eq:star} is still satisfied for $(i',j')<(i,j)$.\medskip

We thus modify the $m_\cdot,n_\cdot,a_\cdot,b_\cdot$ in \eqref{eq:expr} to
get the hypotheses of Lemma \ref{lemmareductionofp}.
\end{proof}

\section{The case of two components}\label{sec:twocomponents}
In  this section, we show how to eliminate non-semi-snc singularities from the strata $\Sigma_{2,q}$. 

Again, $(X,D,E)$ denotes a triple as in Definition \ref{def:triples},
and we use the notation of the latter. As in Theorem \ref{thm:fortriples}, we assume
that $X$ is snc, $D$ is reduced and has no component in the singular locus of $X$, and 
$(X,E)$ is semi-snc.

\begin{proposition}\label{propositioncasepequal2}
Assume that every point of $X$ lies in at most two components of $X$ and that $(X^1,D^1,E^1)$ is semi-snc.
Then there is a sequence of blowings-up with smooth admissible centers such that:
 \begin{enumerate}
 \item Each center of blowing-up consists of only non-semi-snc  points.
 \item For each blowing-up, the preimage of $\Sigma_{2,q}$, for any $q$,
 lies in the union of the $\Sigma_{2,r}$ ($r\leq q$) and the $\Sigma_{1,s}$.
 \item In the final transform of $(X,D,E)$, all points of $\Sigma_{2,q}$ are semi-snc, for every $q$.
 \end{enumerate}
 \end{proposition}

The proof will involve several lemmas. First we show how to blow-up to make $J_a=\mathcal{O}_{X,a}$ at every point $a$. We will use the assumptions of Proposition \ref{propositioncasepequal2} throughout the section. Consider $a\in X$. 
Then $X$ is embedded locally at $a$ in a smooth
variety $Y$ with a system of coordinates $x_1,x_2,y_1,\ldots,y_q, z_1,\ldots,z_{n-q-2}$ in a neighborhood $U$ of $a=0$, in which we can write:
\begin{align*}
 X&=X_1 \cup X_2,\\
D&=D_1+D_2,
\end{align*}
where $X_1=(x_1=0)$, $X_2=(x_2=0)$, $D_1=(x_1=y_1\dotsm y_q=0)$ and $D_2=(x_2=f=0)$, for some $f\in\mathcal{O}_{Y,a}$. This notation will be used
in Lemmas \ref{claimcasepequal21}, \ref{removingJpequals2} and in the proof
of Proposition \ref{propositioncasepequal2} below.

Recall the ideal $J = J(X,D)$ (Definition \ref{definitionJ}) that captures an
important obstruction to semi-snc, see Lemma \ref{lemmafactorssnc}; $J$ is the quotient of the ideal of $D_2\cap X_1$ by that of $D_1\cap X_2$ in $\mathcal{O}_{Y}$. 

Consider $V(J)$ as a hypersurface in $X_1\cap X_2$, and the divisor $D_1|_{X_1\cap X_2}+E|_{X_1\cap X_2}$. We will blow up to get $J=\mathcal{O}_{Y}$ using desingularization of 
$(V(J),\allowbreak D_1|_{X_1\cap X_2}+E|_{X_1\cap X_2})$; i.e., using the desingularization algorithm
for the hypersurface $V(J)$ embedded in the smooth variety $X_1\cap X_2$, with exceptional divisor $D_1|_{X_1\cap X_2}+E|_{X_1\cap X_2}$. The resolution algorithm gives a
sequence of blowings-up that makes the strict  transform of $V(J)$ smooth and snc with respect to the exceptional divisor; we include a final blowing-up of the smooth hypersurface $V(J)$ to make the strict transform empty (``principalization'' of the ideal $J$). It is not necessarily true, however, that $J(X,D)'=J(X',D')$. Therefore, after the preceding blowings-up, we do not necessarily have $J(X',D')=\mathcal{O}_{Y'}$. Additional ``cleaning'' blowings-up will be needed.
 
Example \ref{example2} gives a simple illustration of the problem we resolve in this section. In the example, $V(J)=(x_1=x_2=z=0)$, and our plan is to blow-up with the latter as center $C$ to resolve $J$. In the example, this blowing-up is enough to make $(X,D)$ semi-snc.

\begin{lemma}\label{lem:computeJ2}
Let $R$ denote a regular local ring, and suppose that $x_1,x_2,\allowbreak y_1,\ldots, y_q$ are 
part of regular system of parameters of $R$. Let $f\in R$. Then there exists a maximum
subset $\{i_1<\ldots<i_{t}\}$ of $\{1,\ldots,q\}$ (with respect to inclusion), such that $f$
can be written in the form $f=x_1g_1+x_2g_2+y_{i_1}\dotsm y_{i_t}g_3$. Moreover, 
\[
 [(x_1,x_2,f):(y_1\dotsm y_q)]=(x_1,x_2,g_3).
\]
 \end{lemma} 
 
\begin{proof}
Let $f=x_1g_1+x_2g_2+y_{i_1}\dotsm y_{i_t}g_3$ with $\{i_1,\ldots, i_t\}$ maximal by inclusion, among the subsets of $\{1,\ldots,q\}$. Assume that$f=x_1h_1+x_2h_2+y_{j}h_3$ with $j\notin\{i_1,\ldots,i_t\}$. Then $y_{i_1}\dotsm y_{i_t}g_3\in(x_1,x_3,y_j)$. Since $(x_1,x_2,y_j)$ is prime and $j\notin\{i_1,\ldots,i_t\}$,  we have $g_3\in(x_1,x_2,y_j)$. It follows that there are 
$g'_1,g'_2,g'_3$ such that $f=x_1g'_1+x_2g'_2+y_jy_{i_1}\dotsm y_{i_t}g'_3$, contradicting the maximality of $\{i_1,\ldots,i_t\}$. Therefore $\{i_1,\ldots,i_t\}$ is actually maximum.

For the second part of the lemma: Clearly, $[(x_1,x_2,f):(y_1\dotsm y_q)]\supset(x_1,x_2,g_3)$. Assume that $h \in [(x_1,x_2,f):(y_1\dotsm y_q)]$; i.e.,
$y_1\dotsm y_qh\in(x_1,x_2,f)$. It follows that there is $c\in R$ such that $y_1\dotsm y_qh -y_{i_1}\dotsm y_{i_t}g_3c\allowbreak \in(x_1,x_2)$. Since $(x_1,x_2)$ is prime, we must have $y_{j_1}\dotsm y_{j_{q-t}}h-g_3c\in(x_1,x_2)$, where $\{j_1,\ldots, j_{q-t}\}=\{1,\ldots, q\}\setminus\{i_1,\ldots,i_t\}$. Then $g_3c\in(x_1,x_2,y_{j_k})$, for every $k=1,\ldots,q-t$.

We claim that $g_3\notin(x_1,x_2,y_{j_k})$. In fact, if $g_3\in(x_1,x_2,y_{j_k})$, then there is $\widetilde{g_3}$ such that $f-y_{j_k}y_{i_1}\dotsm y_{i_t}\widetilde{g_3}\in(x_1,x_2)$, contradicting the maximality of $\{i_1,\ldots, i_t\}$. 

Therefore, $c\in(x_1,x_2,y_{j_k})$. So there is $\tilde{c}$ such that $hy_{j_1}\dotsm \widehat{y_{j_k}}\dotsm y_{j_{q-t}}\allowbreak-g_3\tilde{c}\in\allowbreak(x_1,x_2)$, where $\widehat{y_{j_k}}$ means that the term is omitted. By iterating this argument for $k\in\{j_1,\ldots, j_{q-t}\}$ we get $h-g_3\tilde{c}\in(x_1,x_2)$, for some $\tilde{c}$. This implies that $h\in(x_1,x_2,g_3)$, proving that $[(x_1,x_2,f):\allowbreak(y_1\dotsm y_q)]\subset(x_1,x_2,g_3)$.
\end{proof}
 
Given a smooth variety $W$ and a blowing-up $\s: W' \to W$ with smooth center
$C \subset W$, we denote by $I'$ the strict transform by $\s$ of an ideal $I \subset \cO_W$, and by $Z'$ the strict transform of a subvariety $Z \subset W$. (We sometimes use the same notation for the strict transform by a sequence of blowings-up.) We also denote by $f'$ the ``strict transform'' of a function $f \in \cO_{W,a}$, where $a\in W$.
The latter is defined up to an invertible factor at a point $a' \in 
\s^{-1}(a)$; $f':=u^{-d}\cdot f\circ\s$ , where $(u=0)$ defines $\s^{-1}(C)$ at $a'$ and $d$ is the maximum such that $f\circ\s\in(u^d)$ at $a'$. 

\begin{lemma}\label{claimcasepequal21}
Let $\sigma:Y'\rightarrow Y$ denote a blowing-up (or a sequence of blowings-up) 
which is (or are) admissible for $(V(J), D_1|_{X_1\cap X_2}+E|_{X_1\cap X_2})$; i.e., with center(s) in $\Supp \mathcal{O}/J$ and snc with respect to $D_1|_{X_1\cap X_2}+E|_{X_1\cap X_2}$. (For simplicity, we maintain the same notation for the transforms).
Then
\begin{equation}\label{eq:inclusion}
J(X',D')\subset J(X,D)'.
\end{equation}
Moreover, if $J(X,D)'=\mathcal{O}_{Y'}$ and $a'\in X_1\cap X_2$, then $J(X',D')_{a'}=(x_1,x_2,u^{\alpha})$ (in coordinates as above), where $u^{\alpha}$ is a monomial in generators of the ideals of the components of the exceptional divisor
of $\sigma$.
 \end{lemma}
 
\begin{remark} By \eqref{eq:inclusion}, if $J(X,D)'\allowbreak\neq\allowbreak\mathcal{O}_{Y'}$, then $J(X',D')\neq\mathcal{O}_{Y'}$. Therefore, by Lemma \ref{lemmafactorssnc}, we never blow-up semi-snc points of the transforms of $(X,D)$ while desingularizing $J(X,D)$ .
\end{remark} 

 \begin{proof}
Let $I_{X_1}$, $I_{X_2}$, $I_{D_1}$ and $I_{D_2}$ denote the ideals in $\cO_Y$ of $X_1$, $X_2$, $D_1$ and $D_2$ respectively. Locally at $a\in X_1\cap X_2$, we have $I_{X_1}=(x_1)$, $I_{X_2}=(x_2)$, $I_{D_1}=(x_1,y_1\dotsm y_q)$ and $I_{D_2}=(x_2,f)$. Then
\begin{align}\label{eq:J}
J(X,D)&=[I_{X_1}+I_{D_2} : I_{X_2}+I_{D_1}]\\
  &=[(x_1,x_2,f):(x_1,x_2,y_1\dotsm y_q)]\notag\\
  &=[(x_1,x_2,f):(y_1\dotsm y_{q})],\notag
\end{align}
where the last equality follows from the definition of quotient of ideals and the fact that $x_1,x_2\in(x_1,x_2,f)$. 

At a point $a'\in\sigma^{-1}(a)$ with $a'\in X_1'\cap X_2'$,
\begin{align*}
 J(X',D')&=[I_{X_1}'+I_{D_2}' : I_{X_2}'+I_{D_1}']\\
         &=[(x_1')+(x_2,f)' : (x_1',x_2',y_1'\dotsm y_q')]\\
         &=[(x_1')+(x_2,f)' : (y_1'\dotsm y_q')].
\end{align*}

In general, $(I+K)'\supset I'+K'$ and, if $I\supset K$, then $[I:L]\supset[K:L]$, where
$I,K,L$ are ideals. Then
\begin{align*}
 J(X,D)'&=\left(\frac{1}{y_1\dotsm y_q}\left((x_1,x_2,f)\cap(y_1\dotsm y_q)\right)\right)'\\
        &=\frac{1}{y_1'\dotsm y_q'}\left((x_1,x_2,f)'\cap(y_1'\dotsm y_q')\right)\\
        &=[(x_1,x_2,f)':(y_1'\dotsm y_q')]\\
        &=[\left((x_1)+(x_2,f)\right)':(y_1'\dotsm y_q')]\\
        &\supset[(x_1)'+(x_2,f)' : (y_1'\dotsm y_q')]\\
        &=J(X',D').
\end{align*}

Now assume that $J(X,D)'=\mathcal{O}_{Y'}$. Write $f=x_1g_1+x_2g_2+y_{i_1}\dotsm y_{i_t}g_3$ as in Lemma \ref{lem:computeJ2}. The center of the blowing-up lies in 
$\Supp \mathcal{O}_Y/J \subset X_1\cap X_2$ and has normal crossings with
respect to $D_1|_{X_1\cap X_2}+E|_{X_1\cap X_2}$. It follows that
$I_{X_1}'+I_{D_2}'=(x_1',x_2',u^\alpha y_{i_1}'\dotsm y_{i_t}'g_3')$ for $u^\alpha=u_1^{\alpha_1}\dotsm u_t^{\alpha_t}$ a monomial in generators of the ideals of the components of the exceptional divisor. We can then compute
\begin{align}\label{eq:J'}
 J(X',D')&=[I_{X_1}'+I_{D_2}' : (y_1'\dotsm y_q')]\\
         &=[(x_1',x_2',u^\alpha y_{i_1}'\dotsm y_{i_t}' g_3'):(y_1'\dotsm y_q')].\notag
\end{align}
But $J(X,D) = (x_1,x_2,g_3)$, from \eqref{eq:J}.
Since $J(X,D)'= \mathcal{O}_{Y',a'}$ and $a'\in X_1\cap X_2$, it follows that $g_3'$ is a unit. The second assertion of the lemma follows by applying Lemma \ref{lem:computeJ2} to \eqref{eq:J'}.
 \end{proof}

 \begin{lemma}\label{removingJpequals2}
Consider the transform $(X',D',\tE)$ of $(X,D,E)$ by the desingularization of $(V(J),D_1|_{X_1\cap X_2}+E|_{X_1\cap X_2})$. Then:
\begin{enumerate}
\item For every $q$, $\Sigma_{2,q}(X',D')$ lies in the inverse image of $\Sigma_{2,q}(X,D)$.
\item
Let $a'\in X'$. Then the ideal $J(X',D')_{a'}$ is of the form 
$(x_1,x_2,u^\alpha)$, where $X_1'=(x_1=0)$, $X_2'=(x_2=0)$ and $u=u_1^{\alpha_1}\dotsm u_t^{\alpha_t}$ is a monomial in the generators $u_i$ of the ideals of the components of $\tE$. Thus $V(J(X',D'))$ consists of some components of $X_1\cap X_2\cap E$.
\item After a finitely many blowings-up of components of $V(J(X',D'))$ (and its successive transforms), the transform $(X'',D'')$ of $(X,D)$ satisfies $J(X'',D'')=\mathcal{O}_{Y''}$. (For functoriality, the components to be blown up can be chosen according to the order on the components of $E$.
\end{enumerate} 
\end{lemma}

 \begin{proof}
(1) is clear (and is independent of the hypothesis). (2) follows from the second assertion
of Lemma \ref{claimcasepequal21}.

For (3), let us suppose (to simplify notation) that $J(X,D)$ already satisfies the conclusion of (2). Consider the intersection of $X_1$, $X_2$ and the component $H_1$ of the exceptional divisor defined by $(u_1=0)$. We blow-up the irreducible components of this intersection lying inside $\Supp\mathcal{O}/J$. Locally, $X_1 \cap X_2 \cap H_1$
is defined by $(x_1=x_2=u_1=0)$. In the $u_1$-chart, $D_2'=(x_2'=f'=0)$. Since $(x_1,x_2,u^\alpha)=J(X,D)=[(x_1,x_2,f):(y_1\dotsm y_q)]$, we can write $f=x_1g_0+x_2g_1+yu^\alpha$ with $y=y_{i_1}\dotsm y_{i_t}$ as in Lemma \ref{lem:computeJ2}. Therefore, after the blowing-up, $J(X',D')=(x_1',x_2',u_1^{
\beta_1}u_2^{\alpha_1}\dotsm u_t^{\alpha_t})$ with $\beta_1<\alpha_1$ in the $u_1$-chart. In the $x_1$ and $x_2$-charts,
$X_1$ and $X_2$ are moved apart; i.e. we have only strata $\Sigma_{1,k}$ 
(for certain $k$). After a finite number of such blowings-up, we get $J(X',D')=\mathcal{O}_{Y'}$ as wanted.
 \end{proof}

 \begin{proof}[Proof of Proposition \ref{propositioncasepequal2}]
The proof has three steps:
\begin{enumerate}
\item We use Lemma \ref{removingJpequals2} to reduce to the case $J=\mathcal{O}_Y$.
\item Let $r=r(X,D)$ denote the maximum number of components of $D_1$ passing 
through a non-semi-snc point in $X_1\cap X_2$. We make a single blowing-up to reduce $r$. The result will be that $J$ becomes a monomial ideal, as in Lemma \ref{removingJpequals2}(2).
\item We proceed as in Lemma \ref{removingJpequals2}(3) to reduce again to $J=\mathcal{O}_Y$ (without increasing $r$).
\end{enumerate}
Steps (2) and (3) are repeated until the set of non-semi-snc points in $X_1\cap X_2$ is empty. This occurs after finitely many iterations, since $r$ can not decrease indefinitely.\medskip

(1) We begin by applying Lemma \ref{removingJpequals2} to make $J=\mathcal{O}_Y$. 
\smallskip

(2) Assume that $J=\mathcal{O}_Y$. Let $a\in X$. We use a local embedding of $X$ to write
\begin{align*}
 X&=X_1\cup X_2\\
 D&=D_1+D_2,
\end{align*}
where $X_1=(x_1=0)$, $X_2=(x_2=0)$, $D_1=(x_1=y_1\dotsm y_q=0)$, $D_2=(x_2=f=0)$ for some $f\in\mathcal{O}_Y$ (in the notation at the beginning
of the section). By Lemma \ref{lem:computeJ2}, since $J=\mathcal{O}_Y$, we have $f=x_1g_0+x_2g_1+y_1\dotsm y_s$, for some $s\leq q$. Write $f|_{(x_2=0)}=f_1\dotsm f_\ell$, where each $f_i$ is irreducible. We must have $\ell\leq\ord_a(f)\leq s\leq q$; 
therefore $a\in\Sigma_{2,\ell}$. By Lemma \ref{lem:HSlemma}, $H_{\Supp{D},a}=H_{p,\ell}$ if and only if $\ell=q$. Therefore, by Lemma \ref{lemmafactorssnc}, $(X,D,E)$ is semi-snc at $a$ if and only if $\ell=q$. The idea is to blow-up with center given locally by 
$(x_1=x_2=y_1=\ldots=y_q=0)$.

Define 
\[
 C_r:=\Sigma_{1,r}(X_1,D_1)\cap X_2,
\]
where $r = r(X,D)$. Consider a component $Q$ of $C_{r}$ which includes a non-semi-snc point of $(X,D,E)$ in $X_1\cap X_2$. We will prove that $Q$ is closed and consists only of non-semi-snc points of $(X,D,E)$. We will blow-up the union $C$ of all such components of $C_r$.

The set of semi-snc points is open, so the set of semi-snc points in $Q$  is open in $Q$. At a non-semi-snc point $a$ in $Q$, we have a local embedding and coordinates as above in which we can write 
\begin{align*}
 D_1&=(x_1=y_1\dotsm y_r=0)\\
 D_2&=(x_2=f=0),
\end{align*}
where $f|_{(x_2=0)}$ factors into $\ell<r$ irreducible factors; i.e., $D_2$ has $\ell$ irreducible components passing through $a$. In this neighborhood of $a$ in $Q$, 
all points of $Q$ are non-semi-snc. Thus the set of non-semi-snc points is also open in 
$Q$. Since $Q$ is irreducible, it only contains non-semi-snc points. At a point $b$ of 
$\overline{Q}\setminus Q$, the number of components of $D_1$ can be only 
$ > r$. Since $r$ is maximum over the non-semi-snc points, $b$ is a semi-snc point. 
This contradicts the fact that $Q$ contains only non-semi-snc points. Therefore $Q$ is closed.

Thus $C$ is closed and consists of only non-semi-snc points. We can compute locally 
the effect of blowing up $C$. In the $x_1$- and $x_2$-charts, the preimages of a point
$a\in C$ lie in only one component of $X$. In the $y_i$-chart, we compute 
\begin{align*}
 D_1'&=(x_1=y_1\dotsm \widehat{y_i}\dotsm y_r=0)\\
 D_2'&=(x_2=x_1y_i^{j_1}g_0'+x_2y_i^{j_2}g_1'+y_1\dotsm y_i^{j_3}\dotsm y_s=0),
\end{align*}
where $y_i$ is now a generator of the ideal of a component of the exceptional divisor, 
$\widehat{y_i}$ means that the factor is missing from the product, and at least one of $j_1,j_2,j_3$ equals zero. As a result, $r(X',D')<r(X,D)$.
It may happen that $J(X',D')$ is no longer equal to $\mathcal{O}_{Y'}$, but we can 
calculate that $J(X',D')=(x_1,x_2,y_i^{j_3})$ in the $y_i$-chart.
\smallskip

(3) We apply again Lemma \ref{removingJpequals2}(3). The centers of blowing up
are given locally by $X_1\cap X_2 \cap (y_i=0)$.These blowings-up do not increase $r(X,D)$. 
\smallskip

Therefore, after a finite number of iterations, every point lying in two components of $X$ is semi-snc.
\end{proof} 
  

\section{The non-reduced case}\label{sec:fixingmultiplicities}
The previous sections establish Theorem \ref{thm:main} in the case that $D$ is reduced. In this section we describe the blowings-up necessary to deduce the non reduced case.
In other words, we assume that $(X,D_{\text{red}})$ is semi-snc, and we will
prove Theorem \ref{thm:main} under this assumption.
 
The assumption implies that, for every $a\in X$, there is a local embedding in a smooth variety $Y$ with coordinates $x_1,\ldots,x_p,y_1,\ldots,y_q,z_1\ldots,\allowbreak z_{n-p-q}$ in which $a=0$ and
 \begin{align}\label{eq:red}
 X&=(x_1\dotsm x_p=0),\notag\\
 D&=\sum_{(i,j)}a_{ij}(x_i=y_j=0),
 \end{align}
 for some $a_{ij}\in\mathbb{Q}$. Since the reduced pair is semi-snc, we can assume 
that $a_{ij}\neq0$, for every $(i,j)$ in the index set. Nevertheless, the following 
argument if valid even if we allow the possibility that some $a_{ij} = 0$.

The pair $(X,D)$ is semi-snc at $a$ if and only if $a_{ij}=a_{i'j}$ for all $i,i',j$; see Example \ref{ex:multiplicities}. In this section, we transform $D$ by taking only its strict transform $D'$; see Definition \ref{def:transf}. We can neglect the exceptional divisor because, if 
$\sigma:X'\rightarrow X$ is a blowing-up with smooth center simultaneously normal crossings with respect to $X$ and $\Supp D$, then $(X',D_\text{red}'+\Ex(\sigma))$ is
semi-snc provided that $(X,D_\text{red})$ is semi-snc.
Since all components of $\Ex(f)$ appear with multiplicity one, if we make $(X',D')$ semi-snc then $(X',D'+\Ex(\sigma))$ will be semi-snc as well.

We define an equivalence relation on components of $D$ passing through a point of $X$.

\begin{definition}\label{def:equivcompofD}
Let $a\in X$ and let $D_1$, $D_2$ denote components of $D$ passing through $a$.
We say that $D_1$ and $D_2$ are \emph{equivalent} (\emph{at} $a$) if either $D_1=D_2$ or the irreducible component of $D_1\cap D_2$ containing $a$ has codimension $2$ in $X$. 
\end{definition}

Clearly, if $D_1$ and $D_2$ are equivalent at $a$, then $D_1\cap D_2$ is of the
form $(x_1 = x_2 = y_j = 0)$, for some $j$, at $a$, so the irreducible component
of $D_1\cap D_2$ containing $a$ is smooth and $D_1$, $D_2$ are equivalent
at each of it points.

To check that the preceding relation is transitive, let $D_1=(x_{i_1}=y_{j_1}=0)$, $D_2=(x_{i_2}=y_{j_2}=0)$ and $D_3=(x_{i_3}=y_{j_3}=0)$ in coordinates as before. If $D_1$ is equivalent to $D_2$ (at $a=0$), then $j_1=j_2$. If $D_2$ is equivalent to $D_3$,
then $j_2=j_3$. Therefore $D_3$ is equivalent to $D_1$. Reflexivity and symmetry are clear.

Given $a\in X$, let $p(a)$ denote the number of components of $X$ passing through $a$, and let $q(a)$ denote the number of equivalence classes represented by the set of components of $D$ passing through $a$. In local coordinates as before, $q(a)$ is the total number of $j$ for which there exists $a_{ij}\neq0$. Define $\iota:X\rightarrow\mathbb{N}^2$ by $\iota(a):=(p(a),q(a))$. We give $\mathbb{N}^2$ the partial order in which $(p_1,q_1)\geq(p_2,q_2)$ if and only if $p_1\geq p_2$ and $q_1\geq q_2$. Then 
$\iota$ is upper semi-continuous. Therefore, the maximal locus of $\iota$ is a closed set.

Observe that $(X,D)$ is semi-snc at $a$ if and only if $a_{ij}$ is constant on each equivalence class of the set of components of $D$ passing through $a$. Consider the maximal locus of $\iota$. Each irreducible component of the maximal locus of $\iota$ consists only of semi-snc points or only of non-semi-snc points, because all points in one of these irreducible components are contained in the same irreducible components of 
$D$. We blow up with center the union of those components of the maximal locus of 
$\iota$ that contain only non-semi-snc points. In the preimage of the center, $\iota$ decreases. In fact, in local coordinates at a point, as before, we are simply blowing up 
with center
\[
 C=(x_1=\ldots=x_{p(a)}=y_1=\ldots=y_{q(a)}=0).
\]
Therefore, either one component of $X$ is moved away or all components of $D$ in one equivalence class are moved away.

Let $W$ be the union of those components of the maximal locus consisting of semi-snc points. The previous blowing-up is an isomorphism on $W$. So $(X',D')$ is semi-snc on $W' = W$, and therefore in a neighborhood of $W'$. For this reason, the union of the components of the maximal locus of $\iota$ on $X'\setminus W'$ which contain only
non-semi-snc points, is a closed set in $X'$. Therefore, we can repeat the procedure on $X'\setminus W'$.

Clearly, $\mathbb{N}^2$ has no infinite decreasing sequences with respect to the given
order. After the previous blowing-up, the maximal values of $\iota$ on the set of non-semi-snc points of $(X,D)$ decrease. Therefore, after a finite number of iterations of the
procedure above, the set of non-semi-snc becomes empty. 

\begin{remark}
Suppose that $(X,D_{\text{red}})$ is semi-snc (i.e., all $a_{ij}\neq0$ in \eqref{eq:red}, at every point). Then the blowing-up sequence in this section is given simply by the
desingularization algorithm for 
$\Supp D$, but blowing up only those components of the maximal locus of the invariant on the non-semi-snc points.
\end{remark}
\section{Functoriality}\label{sec:functoriality}
In this final section, we make precise and prove the functoriality assertion of Remark \ref{rem:main}.(3).

We say that a morphism $f:Y\rightarrow X$ preserves the number of irreducible components at every point if, for every $b\in Y$, the number of irreducible components of $Y$ at $b$ equals the number of components of $X$ at $f(b)$.

The Hilbert-Samuel function, and in fact the desingularization invariant (beginning with
the Hilbert-Samuel function), is invariant with respect to \'{e}tale morphisms; see
\cite[Remark 1.5]{BMinv}. A smooth morphism $f:Y\rightarrow X$ factors 
locally 
as an \'{e}tale morphism and a projection from a product with an  affine 
space $\mathbb{A}^n$. Therefore, if $f(b)=a$, then 
$H_{Y,b}=H_{X\times\mathbb{A}^n,(a,0)}$ 
and the remaining terms of the invariant 
are the same at $a$ and $b$. To show that the desingularization sequence of 
Theorem \ref{thm:main} is functorial with respect to \'{e}tale (or smooth) morphisms that preserve the number of irreducible components, we just need to show that each blowing-up involved is defined using only the desingularization invariant and the number of components of $X$ and $D$ passing through a point. We can recapitulate each step of the algorithm in Section \ref{sec:maintheorem}:

Step 1 is an application of Theorem \ref{thm:theoremB}. Functoriality of the blowing-up sequence in the latter is proved in \cite{BMmin} and \cite{BDV}. Step 2 is obtained from
the desingularization algorithm of \cite{BMinv} applied to the components of $D$ lying in the intersection of pairs of components of $X$. The blowing-up sequence involved
is functorial with respect to \'{e}tale (or smooth) morphisms in general. 

Step 3, Case A provides a blowing-up sequence completely determined by the Hilbert-Samuel function and the strata $\Sigma_{p,q}$, for $p\geq 3$. The strata $\Sigma_{p,q}$ are defined in terms of the number of components of $X$ and $D$ passing through a point. Step 3, Case B gives a sequence of blowing-up determined by desingularization of the hypersurface $V(J)$ and the number $r(X,D)$ defined in terms of number of components of $D$; see Proposition \ref{propositioncasepequal2}. Step 3, Case C, is again a use of Theorem \ref{thm:theoremB}. 

Finally, the blowings-up of Step 4 are determined by the number of components of $D$ passing through a point and the equivalence relation on the components of $D$ passing through a point, of Definition \ref{def:equivcompofD}. This equivalence relation is preserved by \'{e}tale (or smooth) morphisms.

\begin{remark}\label{rem:funct}
It is not possible to drop the condition on preservation of the number of components in the functoriality statement, for any desingularization that preserves precisely the class of snc singularities. In fact, 
assume that $X$ is nc but not snc at $a$ (see Example \ref{ex:node}). Then there is an \'{e}tale morphism $f:Y\rightarrow X$ such that $Y$ is snc at $b$ and $f(b)=a$. The desingularization must modify $X$ at $a$. It is impossible to pull back this desingularization to $Y$ and still get a desingularization preserving snc because the
latter must be an isomorphism at $b$.
\end{remark}



\end{document}